%% file: prod_vvmf.tex
\documentclass[a4paper,11pt,oneside,headsepline,bibtotoc]{scrartcl}

\input{packages}
\input{layout}

\input{shortcuts}

\input{shortcuts_this}

\newcommand{\headertitle}{{\normalfont%
  Products of Vector Valued Eisenstein Series
}}
\newcommand{\headerauthors}{
  M.~Westerholt-Raum
}

\begin{document}

\begin{plainfootnotes}
\begin{flushleft}
{\fontfamily{lms}\sffamily
  \hspace{20pt}{\huge%
  Products of Vector Valued Eisenstein Series%
  }
}
\\[.6em]\hspace{20pt}{\large%
  Martin Westerholt-Raum%
  \footnote{The author thanks the Max Planck Institute for Mathematics for their hospitality.  The paper was partially written, while he was supported by the ETH Zurich Postdoctoral Fellowship Program and by the Marie Curie Actions for People COFUND Program.}
}
\\[1.2em]
\end{flushleft}
\end{plainfootnotes}

\thispagestyle{scrplain}


{\small
\noindent
{\tbf Abstract:}
We prove that products of at most two vector valued Eisenstein series that originate in level~$1$ span all spaces of cusp forms for congruence subgroups.  This can be viewed as an analogue in the level aspect to a result that goes back to Rankin, and Kohnen and Zagier, which focuses on the weight aspect.  The main feature of the proof are vector valued Hecke operators.  We recover several classical constructions from them, including classical Hecke operators, Atkin-Lehner involutions, and oldforms.   As a corollary to our main theorem, we obtain a vanishing condition for modular forms reminiscent of period relations deduced by Kohnen and Zagier in the context their previously mentioned result.
\\[.35em]
\textsf{\textbf{%
  vector valued Hecke operators%
\hspace{0.3em}{\tiny$\blacksquare$}\hspace{0.3em}%
 period relations%
\hspace{0.3em}{\tiny$\blacksquare$}\hspace{0.3em}%
 cusp expansions of modular forms%
}}
\\[0.15em]
\noindent
\textsf{\textbf{%
  MSC Primary:
  11F11%
\hspace{0.3em}{\tiny$\blacksquare$}\hspace{0.3em}%
  MSC Secondary:
  11F67
}}
}


\vspace{-1.5em}
\renewcommand{\contentsname}{}
\setcounter{tocdepth}{2}
\tableofcontents
\vspace{1.5em}


\Needspace*{4em}
\addcontentsline{toc}{section}{Introduction}
\markright{Introduction}
\lettrine[lines=2,nindent=.2em]{\tbf I}{n} a paper from 1951, published in 1952, Rankin~\cite{rankin-1952} derived an expression in terms of periods for scalar products $\langle E_l E_{k-l}, f \rangle$ for cuspidal Hecke eigenforms~$f$ of weight~$k$.  This served Kohnen and Zagier~\cite{kohnen-zagier-1984} in connecting modular forms with period polynomials. As an immediate consequence of their work one concludes that products of at most two Eisenstein series span all spaces of modular forms for~$\SL{2}(\ZZ)$.  In fact, it is possible to describe linear relations of the $E_l E_{k-l}$ by period polynomials~\cite{skoruppa-1993}.

Let us write $\rmM_k$ for the space of level~$1$, weight~$k$ modular forms, and $\rmE_k$ for the subspace of Eisenstein series. The aforementioned consequence of Kohnen's and Zagier's result reads as follows:
\begin{gather}
\label{eq:introduction-kohnen-zagier}
  \rmM_k
=
  \rmE_k
  +
  \lspan_{4 \le l \le k - 4} \rmE_l \cdot \rmE_{k-l}
\text{.}
\end{gather}
This was considered by Imamo\u glu and Kohnen~\cite{imamoglu-kohnen-2005} in the case of $\Gamma_0(2)$, and also generalized to $\Gamma_0(p) \subset \SL{2}(\ZZ)$ by Kohnen and Martin~\cite{kohnen-martin-2008}.  A similar statement, based on products of at most two ``toric'' modular forms, was found by Borisov and Gunnells in~\cite{borisov-gunnells-2003}.

Note that we can view \eqref{eq:introduction-kohnen-zagier} as a statement about products of Eisenstein series focusing on the weight aspect.  We study the level aspect of the same problem:  Can cusp forms of any level be expressed as products of at most two Eisenstein series, varying the level but fixing the weight?  The affirmative answer, which we provide, is most naturally phrased in terms of vector valued modular forms.  For a complex representation $\rho$ of $\SL{2}(\ZZ)$ and $k \in \ZZ$ write $\rmM_k(\rho)$ for the space of all vector valued modular forms of weight~$k$ and type~$\rho$.  By definition, $f \in \rmM_k(\rho)$ satisfies
\begin{gather*}
  f\big( \frac{a \tau + b}{c \tau + d} \big)
=
  \rho\big(\!\! \left(\begin{smallmatrix} a & b \\ c & d \end{smallmatrix}\right) \!\!\big)\,
  (c \tau + d)^k\, f(\tau)
\end{gather*}
for all $\left(\begin{smallmatrix} a & b \\ c & d \end{smallmatrix}\right) \in \SL{2}(\ZZ)$.  A corresponding subspace $\rmE_k(\rho)$ of Eisenstein series can be defined in a natural way, which is explained in Section~\ref{sec:eisenstein-series}.  For every $N$, we will define in Section~\ref{sec:hecke-operators} a vector valued Hecke operator $T_N$.  It yields a map $T_N :\, \rmE_k(\bbone) \ra \rmE_k(\rho_{T_N})$, also described in Section~\ref{sec:hecke-operators}, where $\bbone$ is the trivial representation of~$\SL{2}(\ZZ)$ and $\rho_{T_N}$ in the easiest case is the permutation representation of $\SL{2}(\ZZ)$ on cosets $\Gamma_0(N) \backslash \SL{2}(\ZZ)$.  In genearl, $\rho_{T_N} = T_N\, \bbone$ defined in Section~\ref{ssec:hecke-operators-on-representations}. For the time being, it suffices to know that the components of $T_N\,\rmE_k(\bbone)$ can be expressed in terms of $E_k\big( \frac{a \tau + b}{d} \big)$, where $\left(\begin{smallmatrix} a & b \\ 0 & d \end{smallmatrix}\right) \in \GL{2}(\QQ)$.  As a further ingredient to formulate our Main Theorem, note that by composition, homomorphisms~$\phi:\, \rho \ra \sigma$ of representations give rise to maps between spaces of modular forms $\phi:\, \rmM_k(\rho) \ra \rmM_k(\sigma)$.
\begin{maintheorem}
\label{thm:maintheorem}
Let $\rho$ be a complex representation of $\SL{2}(\ZZ)$ whose kernel contains a congruence subgroup.  For even integers $k \ge 8$ and $l$ such that $l, k-l \ge 4$,  we have
\begin{gather}
\label{eq:product-of-eisenstein-series}
  \rmM_k(\rho)
=
  \rmE_k(\rho)
  +
  \lspan_{\substack{0 < N, N' \in \ZZ \\ \phi :\, \rho_{T_N} \otimes \rho_{T_{N'}} \ra \rho}}
  \phi\Big(\;  T_N\,\rmE_{l}(\bbone) \otimes T_{N'}\,\rmE_{k - l}(\bbone) \; \Big)
\,\text{.}
\end{gather}
The sum runs over homomorphisms~$\phi$ from $\rho_{T_N} \otimes \rho_{T_{N'}}$ to $\rho$ for positive integers~$N$ and~$N'$.
\end{maintheorem}
\begin{mainremarkenumerate}
\item
The range of $N$ and $N'$ can be bounded by means of Hecke theory. For example, consider the case that the representation $\rho$ corresponds to the subgroup $\Gamma_0(M) \subseteq \SL{2}(\ZZ)$ for some positive $M \in \ZZ$. That is consider the case that $\rho = \Ind_{\Gamma_0(M)}^{\SL{2}(\ZZ)}\, \bbone$ using notation of Section~\ref{ssec:preliminaries:induced-representations}. The Hecke algebra, viewed as a subalgebra of the endomorphism algebra of $\rmM_k(\Gamma_0(N))$, is finitely generated, say by (classical) Hecke operators $T_N$ with $N \le N_0$. Then it suffices to let $N$ and $N'$ in~\eqref{eq:product-of-eisenstein-series} run through integers between~$1$ and $N_0$. Precise estimates of how many Hecke operators are required to obtain the all of $\rmM_k(\rho)$ are not yet available, but planned for a sequel on computational aspects.

\item
By introducing holomorphic projections of products of almost holomorphic Eisenstein series, we could lower the bound on $k$ to $k \ge 4$ and the bounds on $l$ and $k-l$ to $l, k-l \ge 2$.  This also will be discussed in more detail in a sequel on computational aspects.  Derivatives of Eisenstein series, which are directly related to almost holomorphic modular forms, were already introduced into the subject at the end of Section~2 in~\cite{kohnen-zagier-1984}.

\item
It is currently not clear whether the extra space of Eisenstein series in~\eqref{eq:product-of-eisenstein-series} is needed. This can be directly related to a purely representation theoretic question. Using notation from Section~\ref{sec:eisenstein-series}, the key question is whether for arbitrary representations~$\rho$ with congruence subgroup kernel we have
\begin{gather*}
  V(\rho)(1)_T
=
  \lspan_{\substack{0 < N, N' \in \ZZ \\ \phi :\, \rho_{T_N} \otimes \rho_{T_{N'}} \ra \rho}}
  \phi\Big(\; T_N\, V(\bbone) \otimes T_{N'}\, V(\bbone) \;\Big)
\,\tx{.}
\end{gather*}
\end{mainremarkenumerate}

\subsubsection*{Hecke operators}

For simplicity, assume that $f$ is a level~$1$ modular form of weight~$k$. Classical Hecke operators can be written as a sum $N^{\frac{k}{2}-1} \sum_{m \in \Delta_N} f \big|_k\, m$, where~$\Delta_N$ is the set of all $\left(\begin{smallmatrix} a & b \\ 0 & d \end{smallmatrix}\right)$ with $ad = N$ and $0 \le b < d$.  Instead of summing over $\Delta_N$, we introduce vector valued modular forms, separating contributions of each $m \in \Delta_N$. The vector valued modular form $T_N\, f$ takes values in the $\CC$ vector space with basis $\Delta_N$. Its $m$\thdash\ component for $m \in \Delta_N$ is defined as~$f \big|_k\, m$. We show in Section~\ref{ssec:hecke-operators-on-modular-forms} that this is a vector valued modular form.  Vector valued Hecke operators can be applied to any vector valued modular form. We will defined them in two steps. Given a representation~$\rho$, we find another one~$T_M\,\rho$. Modular forms of type~$\rho$, we show, are mapped to modular forms of type~$T_M\,\rho$.

\paragraph{Compatibility with products}

A crucial property of vector valued Hecke operators, compared to classical ones, is that $T_N\, (f \otimes g)$ can be recovered from $(T_N\, f) \otimes (T_N\, g)$, while $fg \big| T_N \ne ( f \big| T_N ) ( g \big| T_N )$ in general for scalar valued modular forms---cf.~\cite{duke-1999} for a related and amusing topic. Explicitly, the sums $\sum_{m,m'} \big( f\big|_k\,m \big) \big( g\big|_l\,m' \big)$ and $\sum_m fg \big|_{k+l}\,m$ are not equal in any formal basis. Exceptional equalities can occur, but this is forced by dimensions. On the other hand, for a single~$m \in \Delta_N$, we have $\big( f \big|_k\, m \big) \big( g \big|_l\, m \big) = fg \big|_{k+l}\, m$. The tensor product of $T_N\,f$ and $T_N\,g$ has components $\big( f \big|_k\, m \big) \big( g \big|_l\,m' \big)$ for $m,m' \in \Delta_N$. Picking the components with $m = m'$ we obtain $T_N\, (fg)$.

\paragraph{Relation to classical Atkin-Lehner-Li theory}

Atkin-Lehner-Li theory is governed by classical Hecke operators and two additional families of operators. The Atkin Lehner involutions~$W_M$ for $M \isdiv N$, in the simplest case $M = N$, map a modular form~$f$ to a suitable scalar multiple of $\tau^{-k} f(-1\slash N\tau)$. The oldform operator, tentatively denoted by~${\rm sc}_M$ (notation is alluding to reSCaling) for a positive integer~$M$, maps $f$ to $f(M \tau)$. Both constructions can be phrased in terms of elements of $\Delta_M$. We will find a suitable $m \in \Delta_M$ such that $f \big|_k\,m = W_M(f)$, and another $m \in \Delta_M$ such that $f \big|_k\, m = {\rm sc}_M(f)$. In other words, classical Atkin-Lehner-Li theory is subsumed by vector valued Hecke operators.

We reformulate this observation: For simplicity let us fix a congruence subgroup $\Gamma = \Gamma_0(N) \subseteq \SL{2}(\ZZ)$. There is a representation $\rho_\Gamma$ so that $\rmM_k(\Gamma)$, the space of modular forms of weight~$k$ for the group~$\Gamma$, is canonically isomorphic to~$\rmM_k(\rho_\Gamma)$. This isomorphism, here and later, will be denoted by~$\Ind :\, \rmM_k(\Gamma) \ra \rmM_k(\rho_\Gamma)$. The Hecke operators and the Atkin-Lehner involutions are maps from~$\rmM_k(\Gamma)$ to itself. The oldform constructions yield maps from $\rmM_k(\Gamma')$ to $\rmM_k(\Gamma)$ where $\Gamma \subset \Gamma' \subseteq \SL{2}(\ZZ)$. In Proposition~\ref{prop:components-of-hecke-induced} and Remark~\ref{rm:components-of-hecke-induced:projections} of Section~\ref{ssec:hecke-operators-known-constructions}, we give explicit maps
\begin{gather*}
  \pi_{\rm Hecke} :\, T_M\, \rho_\Gamma \lra \rho_\Gamma
\tx{,}\quad
  \pi_{\rm AL} :\, T_M\, \rho_\Gamma \lra \rho_\Gamma
\tx{,}\quad
  \pi_{\rm old} :\, T_M\, \rho_\Gamma \lra \rho_\Gamma
\end{gather*}
such that
\begin{gather*}
  \Ind \circ (\big|_k\,T_M) = \pi_{\rm Hecke} \circ T_M \circ \Ind 
\tx{,}\quad
  \Ind \circ W_M = \pi_{\rm AL} \circ T_M \circ \Ind 
\tx{,}\;\;\tx{and}\quad
  \Ind \circ \alpha_M = \pi_{\rm old} \circ T_M \circ \Ind 
\tx{.}
\end{gather*}
The maps $\pi_{\rm Hecke}$, $\pi_{\rm AL}$, and $\pi_{\rm old}$ intertwine the induction map and the three families of operators.

\paragraph{Relation to other vector valued Hecke operators}

Weil representations are not in the scope of the present paper. Nonetheless, some words seem due: Vector valued modular forms for Weil representations and one certain kind of Hecke operators for them, denoted by~$\uparrow_H^A$ and $\downarrow_H^A$, were considered in~\cite{scheithauer-2011}. The same operators are foundational to the newform theory that Bruinier develops in~\cite{bruinier-2012} for Weil representation attached to cyclic discriminant modules. A vector valued Hecke operator $V_l$ was derived in~\cite{raum-2012c} by means of Jacobi forms. All these operators should relate to the Hecke operators in the present paper. The operator $\uparrow_H^A$ is analogous to the intertwining property described in Lemma~\ref{la:identity-in-hecke-operator-square}. The operator $\downarrow_H^A$ should be its adjoint. Finally, $V_l$ seems to be the same as $T_l$, after applying the theta decomposition and a suitable change of basis. We suggest to consider these claims in the context of a newform theory for vector valued modular forms for Weil representations---very much in the spirit of~\cite{bruinier-2012}.

Hecke operators also previously appeared in work by Bruinier and Stein~\cite{bruinier-stein-2010}. Note that in that paper half-integral weight modular forms occur, while we treat integral weight modular forms. One can extend our construction to modular forms for representations of the metaplectic double cover of~$\SL{2}(\ZZ)$, but the details remain to be worked out. The question in~\cite{bruinier-stein-2010} can be rephrased as whether for half-integral weights and Weil representations it is possible to find an analogue to the map~$\pi_{\rm Hecke}$ above. When put in this way, it is a purely representation theoretic question. It was established in~\cite{bruinier-stein-2010} that there is such a map when working with projective representations, but in general there is a non-trivial obstruction to lifting it to an actual representation.

\paragraph{Relation to automorphic representations}

We sketch a relation between vector valued modular forms and automorphic representation theory. Our description of their links, condensed to half a page, is doomed to be imprecise, since we now invoke the machinery of adelic automorphic representation theory. Nevertheless, it yields one fruitful way to think about vector valued Hecke operators. We work over~$\QQ$ so that at every finite place $v$ corresponds to a prime in the classical sense.

We start by relating vector valued modular forms and automorphic representations. Let us fix a newform~$f$, and consider the associated adelic automorphic representation~$\varpi_f = \bigotimes'_v \varpi_{f,v}$. At each finite place~$v$, we obtain a vector~$w_{f,v} \in \varpi_{f,v}$ that corresponds to~$f$. There is a subgroup~$K \subseteq \GL{2}(\ZZ_p)$ such that the space of $K$ fixed vectors in $\varpi_{f,v}$ is 1\nbd dimensional, spanned by~$w_{f,v}$.

Conjugating any such~$K$ by elements $\gamma$ of~$\GL{2}(\ZZ_p)$, we obtain other vectors $ \gamma w_{f,v} \in \varpi_{f,v}$ that are fixed by $\gamma K \gamma^{-1}$. Since $K \backslash \GL{2}(\ZZ_p)$ is a finite set, one can assemble all $\gamma w_{f,v}$ for $\gamma \in K \backslash \GL{2}(\ZZ_p)$ to a single vector of finite length. From this perspective, vector valued modular forms are a way to view all $K$ with $\dim\, \varpi_{f,v}^K = 1$ simultaneously.

For simplicity, let us assume that $\varpi_{f,v}$ is spherical, that is~$K = \GL{2}(\ZZ_p)$. The vector valued Hecke operator applied to $\Ind(f)$ for a newform~$f$ exposes further parts of $\varpi_{f,v}$. It is standard in context of strong approximation to view the previously defined set $\Delta_N$ for a $p$-power~$N$ as a subset of representatives of $\GL{2}(\ZZ_p) \backslash \GL{2}(\QQ_p)$. Define $K_N$ as the intersection of all $\gamma K \gamma^{-1}$ for $\gamma \in \Delta_N$. Then $T_N\, \Ind(f)$, at the place~$v$, is the same as the previously discussed vector of $K_N$ fixed vectors of $\varpi_{f,v}$. The dimension of $\varpi_{f,v}^{K_N}$ increases as $N$ growth, and this is the reason why $T_N\, \Ind(f)$ has increasingly many components.

From this perspective, it becomes immediately clear, why classical Atkin-Lehner-Li theory must be encoded in vector valued Hecke operators.

\subsubsection*{Period relations}

A proof of~\eqref{eq:introduction-kohnen-zagier} based on~\cite{kohnen-zagier-1984} relies on an explicit formulation of the Eichler Shimura isomorphism in terms of period polynomials.  In particular, linear relations between the $E_l E_{k-l}$ can be rephrased in terms of $L$-values.  As we will explain later, the proof of our main theorem rests on completely different observations.  It is possible to give a vanishing criterion for cusp forms as a corollary to our main theorem.  However, the resulting convolution $L$-series, given in Corollary~\ref{cor:strange-vanishing-condition}, can not be simplified in all cases. The reason is that the expansion of newforms at cusps not mapped to $\infty$ by Atkin-Lehner involutions are not necessarily newforms.  In other words, given a level~$N$ newform~$f$ and a cusp $\frakc$ for which there is no $M \isdiv N$ with $W_M\,\frakc = \infty$, then there is no a priori reason why the expansion of~$f$ at~$\frakc$ is even a Hecke eigenform; let alone a newform.

We refer the reader to the explicit vector valued Eichler Shimura theorem proved by Pa\c sol and Popa~\cite{pasol-popa-2013} in analogy with the level~$1$ case.  It would be interesting to relate our Main Theorem to a statement about periods and compare it to Corollary~\ref{cor:strange-vanishing-condition}. For the time being, we present the latter as curiosity, for which we neither have an application nor a satisfying explanation.

\subsection*{Computing modular forms and their cusp expansions}

Our Main Theorem allows us to span spaces of cusp forms by products of at most two Eisenstein series.  This provides an alternative to algorithms based on modular symbols~\cite{manin-1972,cremona-1997}.  In a sequel to this paper, we will report on how performance of these two methods compare.

Fix a modular form $f$ for a Dirichlet character~$\chi$ mod~$N$.  In the context of, for example, Borcherds products~\cite{borcherds-1998} it is an interesting question how Fourier expansions of~$f$ at cusps of $\Gamma_0(N) \backslash \HS$ can be computed.  If $N$ is square free then modular symbols provide us with a satisfactory theory, which is employed, e.g., in Sage~\cite{sage-6-3}.  Indeed, Atkin Lehner operators act on modular symbols, and in this case permute cusps transitively.

The case of non square free level~$N$ is more complicated. Atkin Lehner involutions fail to act transitively on cusps, and hence modular symbols are insufficient to obtain all cusp expansions. In~\cite{raum-2012c}, we found an algorithm to compute them, which however does not perform very well in practice. Its runtime is subexponential, but not polynomial with respect to the level.  Combining our Main Theorem with formulas for cusp expansions of Eisenstein series obtained in~\cite{weisinger-1977}, we obtain another promising approach to computing cusp expansions. It has the charm of relying only on multiplication of power series and computation of tensor products. The former is straightforward, and the latter can be quickly reduced to linear algebra over cyclotomic fields, which for example is implemented efficiently in the computer algebra system Magma.

\paragraph*{Runtime}
Let us sketch how to compare the runtime of the best known algorithms for computing Fourier expansions of modular forms~\cite{edixhoven-couveignes-2011} and an algorithm which builds on the present work. Writing $\rmM_\bullet(\chi)$ for the $\rmM_\bullet$\nbd module of modular forms for a Dirichlet character~$\chi$. To compute Fourier expansions for all elements of~$\rmM_\bullet(\chi)$ it suffices to determine those of a minimal set of generators. The dimension formula for modular forms implies that their weight is less than~$12$. We therefore consider computations of modular forms for a fixed weight~$k$. We bound the time required to compute all Fourier coefficients $c(n)$ for $n < n_0$ with $n_0 \ra \infty$ of all modular forms of level~$N < N_0$ with $N_0 \ra \infty$. This is justified at least with an eye to combinatorial applications---see~\cite{andrews-1974} and work following the ideas described there.

Let us first assume that we have uniquely identified a modular form. The algorithm by Couveignes and Edixhoven~\cite{edixhoven-couveignes-2011} essentially requires a piece of the Jacobian that corresponds to an irreducible Galois representation. Given such a representation, it computes the $n$\thdash\ Fourier coefficient of the attached newform in ${\rm TIME}(\log(n))$. We are suprressing issues with mod~$l$ Galois representations for the benefit of~\cite{edixhoven-couveignes-2011}. Repeating this for all $n < n_0$ yields~${\rm TIME}(n_0 \log(n_0))$. An algorithm based on the present paper would require for the computation of Fourier coefficients of a modular form an expression in terms of Eisenstein series and vector valued Hecke operators. Given such a representation the remaining computations are to evaluate Fourier coefficients of level~$1$ Eisenstein series, and to multiply Fourier expansions of length~$n_0$. Both yield~${\rm TIME}(n_0 \log(n_0))$. Summing up, in the aspect of computing all Fourier coefficients of a given modular form up to a given precision the two algorithms perform asymptotically equally fast.

We have to discuss how to obtain a Galois representation, and how to evaluate~\eqref{eq:product-of-eisenstein-series}. Galois representations can be computed using modular symbols~\cite{stein-2007}. Computing modular symbols for a Dirichlet character modulo~$N$ relies on computing the reduced echelon normal form of a size~$O(k N)$ matrix with entries in a cyclotomic field of order~$N$ (and degree~$\varphi(N)$). In the present consideration, we neglect utilization of degeneracy maps, whose impact on runtime to the author's knowledge has not been effectively estimated. Let $\omega \in \RR$ be the runtime exponent for computing reduced echelon normal forms. Then computing modular symbols for one fixed character yields ${\rm TIME}(k^\omega N^{2 \omega})$. Summing over all characters modulo~$N < N_0$ yields runtime~${\rm TIME}( k^\omega N_0^{2 \omega + 2} )$.

We pass to~\eqref{eq:product-of-eisenstein-series}. Using Sturm bounds and inspecting the proof of the main theorem it seems plausible to assume that it suffices to decompose $T_{N}\,\bbone \otimes T_{\sqrt{N}}\,\bbone$ to find an expression for level~$N$ modular forms for all divisors of~$N$---a precise analysis would be involved and is planned for the aforementioned sequel on practical aspects. Using the MeatAxe algorithm (cf.~Chapter 7.4 of~\cite{holt-eick-obrien-2005}), we can reduce computations to the echelon form of a matrix of size $O(N^{9\slash 4})$ with coefficients in an order~$N$ cyclotomic field. This yields a runtime estimate ${\rm TIME}( k + N_0^{9 \omega \slash 4 + 2} )$, where the additional term~$k$ stems from the contribution of the weight to Sturm bounds.

The FLINT: Fast Library for Number Theory uses the Strassen algorithm, yields $2.8 < \omega < 3$, while theoretical estimates show that $\omega < 2.4$ is possible. In any case, our implementation with the stated assumption would be slower by order~$\omega \slash 4$. One should Keep in mind that the decomposition of $T_{N}\,\bbone \otimes T_{\sqrt{N}}\,\bbone$ can probably be sped up using for example the well-known classification of representations of $\SL{2}(\QQ_p)$.

\subsection*{Method of proof}

We discuss the method of proof, comparing it to~\cite{kohnen-zagier-1984}.  In the classical setting, one fixes a cusp form~$f$ that is orthogonal to all $E_l E_{k-l}$. Relating this to vanishing of periods one concludes that the period polynomial of an attached cusp form $f'$, which in general is different from~$f$, vanishes.  This, in particular, makes use of an explicit Eichler Shimura correspondence.  In our case, we achieve to reduce ourselves to the case of newforms~$f$, such that the necessary vanishing statement is straightforward.

Reduction to newforms for $\Gamma_0(N) \subseteq \SL{2}(\ZZ)$ is performed by vector valued Hecke operators~$T_N$, introduced in Section~\ref{sec:hecke-operators}.  We show that $\lspan_\phi \phi\big(T_N\,\rmE_l(\bbone) \otimes T_{N'}\, \rmE_{k-l}(\bbone) \big)$ in~\eqref{eq:product-of-eisenstein-series} yields a Hecke module in the classical sense.  If an irreducible Hecke module which appears in $\rmM_k\big(\Gamma_0(N)\big)$, the space of modular forms for $\Gamma_0(N)$, is missed by our construction, then we fix a suitable modular form in there.  More specifically, we may choose a newform~$f$, and are hence reduced to this case.

In the course of the proof, we exploit the symmetry of \eqref{eq:product-of-eisenstein-series} to focus on the case $l \le k \slash 2$. This is necessary to conclude that the $L$-series $L(f, k - l)$ that appears in the proof is nonzero. The case $l = k \slash 2$ is special, since $L(f, k\slash 2)$ may indeed vanish. Note that this difficulty has already appeared in~\cite{borisov-gunnells-2001b}, and after the first preprint version of the present paper in~\cite{dickson-neururer-2015}. Vector valued Hecke operators allows us to circumvent it. We employ them to show that $L(f \times \chi, k\slash 2)$ vanishes for certain Dirichlet characters~$\chi$ and then make use of Waldspurger's  nonvanishing results.

\subsection*{Acknowledgment}
The author is grateful to one of the referees for extraordinarily extensive and very helpful comments, improving readability of this paper.


\section{Preliminaries}

\subsection{Classical modular forms}

As usual, write $\HS$ for the Poincar\'e upper half plane $\{\tau \in \CC \,:\, \Im\tau > 0 \}$.  It carries an action of $\SL{2}(\RR)$ by M\"obius transformations.  Writing $\gamma = \left(\begin{smallmatrix} a & b \\ c & d \end{smallmatrix}\right)$ for an element of $\SL{2}(\RR)$, we have $\gamma \tau = \frac{a \tau + b}{c \tau + d}$.  This action gives rise to a family $|_k$ of actions on holomorphic functions~$f$ on $\HS$, which is indexed by $k \in \ZZ$.  We set $(f\big|_k\,\gamma)(\tau) = (c \tau + d)^{-k}\, f(\gamma \tau)$.  For a character $\chi$ of $\Gamma \subseteq \SL{2}(\ZZ)$ and $\gamma \in \Gamma$, we set $(f\big|_{k, \chi}\,\gamma)(\tau) = (c \tau + d)^{-k}\chi(\gamma^{-1}) \, f(\gamma \tau)$.
\begin{definition}
Let $\Gamma \subseteq \SL{2}(\ZZ)$ be a finite index subgroup and $\chi$ a character of $\Gamma$.  We call a holomorphic function~$f :\, \HS \ra \CC$ a modular form of weight~$k$ for $\Gamma$ and $\chi$ if
\begin{enumerate}[label=(\roman*)]
\item
for all $\gamma \in \Gamma$, we have $f \big|_{k, \chi}\, \gamma = f$, and

\item
for every $\gamma \in \SL{2}(\ZZ)$, $\big(f \big|_k\, \gamma\big)(\tau)$ is bounded as $\tau \rightarrow i \infty$.
\end{enumerate}

\noindent
The space of such functions is denoted by $\rmM_k(\Gamma, \chi) = \rmM_k(\chi)$ and $\rmM_k(\Gamma)$, if $\chi$ is trivial.  The corresponding subspaces of cusp forms are denoted by $\rmS_k(\Gamma, \chi) = \rmS_k(\chi)$ and $\rmS_k(\Gamma)$.
\end{definition}

Denoting the entries of $\gamma \in \SL{2}(\ZZ)$ by $a(\gamma)$, $b(\gamma)$, $c(\gamma)$, and $d(\gamma)$, we define subgroups
\begin{align*}
  \Gamma_0(N)
&=
  \big\{ \gamma \in \SL{2}(\ZZ) \,:\, c(\gamma) \equiv 0 \pmod{N} \big\}
\tx{,}\\
  \Gamma_1(N)
&=
  \big\{ \gamma \in \SL{2}(\ZZ) \,:\, c(\gamma) \equiv 0,\, a(\gamma) \equiv d(\gamma) \equiv 1 \pmod{N} \big\}
\tx{,}\\
  \Gamma_1(N^2,N)
&=
  \big\{ \gamma \in \SL{2}(\ZZ) \,:\, c(\gamma) \equiv 0 \pmod{N^2},\, a(\gamma) \equiv d(\gamma) \equiv 1 \pmod{N} \big\}
\tx{,}\quad\tx{and}\\
  \Gamma(N)
&=
  \big\{ \gamma \in \SL{2}(\ZZ) \,:\, b(\gamma) \equiv c(\gamma) \equiv 0,\, a(\gamma) \equiv d(\gamma) \equiv 1 \pmod{N} \big\}
\tx{.}
\end{align*}
Dirichlet characters mod~$N$ give rise to characters of $\Gamma_0(N)$ by means of $\chi(\gamma) = \chi(d(\gamma))$.  The standard parabolic subgroup of $\SL{2}(\ZZ)$ is defined as $\Gamma_\infty = \big\{ \gamma \in \SL{2}(\ZZ) \,:\, c(\gamma) = 0 \big\}$.

\subsection{Vector valued modular forms}

A (complex) representation $\rho$ of $\Gamma \subset \SL{2}(\ZZ)$ is a group homomorphism from $\Gamma$ to $\GL{}\big( V(\rho) \big)$ for a complex vector space~$V(\rho)$, which is called the representation space of~$\rho$. We focus on finite dimensional representations, in which case the dual of $\rho$, denoted by~$\rho^\vee$, can be defined by $\rho^\vee(\gamma) = \rT \rho(\gamma)^{-1}$. Its representation space~$V(\rho^\vee)$ is the space of linear functionals~$V(\rho)^\vee$ on $V(\rho)$. The trivial representation of any group is denoted by~$\bbone$, and the group will be clear from the context. A $\Gamma$\nbd representation~$\rho$ is called irreducible if the only~$\Gamma$ invariant subspaces of $V(\rho)$ are $V(\rho)$ and $\{0\}$. Given a representation~$\rho$ and an arbitrary irreducible one~$\rho'$, we call $\rho(\rho') := V(\rho') \otimes \Hom(\rho', \rho) \subseteq V(\rho)$ the $\rho'$\nbd isotypical component of~$\rho$. It is the maximal subrepresentation of $\rho$ for which $\Hom(\rho'',\rho) \ne \{0\}$ implies that an irreducible $\rho''$ is isomorphic to $\rho'$. In particular, if $\rho'$ is not a subrepresentation of $\rho$, then $\rho(\rho') = \{0\}$. The isotrivial component of a representation $\rho$ is defined as $\rho(\bbone)$. Recall that in our setting all unitary representations are completely reducible as a direct sum, and to every inclusion $\rho \hra \sigma$ there is a corresponding projection $\sigma \thra \rho$ and vice versa.

For any $k \in \ZZ$ and any finite dimensional representation~$\rho$ of $\SL{2}(\ZZ)$, we set
\begin{gather*}
  \big( f\big|_{k, \rho}\, \gamma \big)(\tau)
=
  (c\tau + d)^{-k} \rho(\gamma^{-1}) f(\gamma \tau)
\text{.}
\end{gather*}

\begin{definition}
Fix a finite dimensional complex representation $\rho$ of $\SL{2}(\ZZ)$.  A holomorphic function $f :\, \HS \ra V(\rho)$, we say, is a modular form of weight~$k$ and type~$\rho$, if
\begin{enumerate}[label=(\roman*)]
\item
for all $\gamma \in \SL{2}(\ZZ)$, we have $f \big|_{k, \rho}\, \gamma = f$, and

\item
for every $v \in V(\rho)^\vee$, $\big( v \circ f \big)(\tau)$ is bounded as $\tau \rightarrow i \infty$.
\end{enumerate}

\noindent
The space of such functions is denoted by $\rmM_k(\rho)$.  A cusp forms is a modular form $f \in \rmM_k(\rho)$ satisfying $\big( v \circ f \big)(\tau) \ra 0$ as $\tau \rightarrow i \infty$ for all $v \in V(\rho)^\vee$. The space of cusp forms is denoted by $\rmS_k(\rho)$.
\end{definition}
\begin{remarkenumerate}
\item
We have defined vector valued modular forms for all finite dimensional, complex representations of $\SL{2}(\ZZ)$.  Our main interest, however, lies in modular forms for congruence subgroups.  This is why we will, from here on, restrict ourselves to representations with finite index kernel. Since such $\rho : \SL{2}(\ZZ) \ra \GL{}\big( V(\rho) \big)$ factor through a finite quotient $\SL{2}(\ZZ) \slash \ker\,\rho$, they are unitarizable. In particular, we can and will assume that $\rho$ is unitary with respect to a scalar product, say, $v,w \ra \langle v,\, w \rangle = \langle v,\, w \rangle_\rho$ on $V(\rho)$.

\item
Representations of $\SL{2}(\ZZ)$ whose kernel is a congruence subgroup, can be factored as a tensor product $\bigotimes_p \rho_p$ of representations whose level is a $p$\nbd power.  Such representations were studied in~\cite{silberger-1970}. In particular, a complete classification is available.
\end{remarkenumerate}

Given $v \in V(\rho)$, we write $\langle f,\, v \rangle$ for the function  $\tau \mapsto \langle f(\tau),\, v \rangle$, which is a modular form for $\ker\,\rho$.  For two functions $f :\, \HS \ra V(\rho)$ and $g :\, \HS \ra V(\sigma)$, we write $f \otimes g :\, \HS \ra V(\rho \otimes \sigma)$ for $(f \otimes g)(\tau) = f(\tau) \otimes g(\tau)$.  Morphisms of representations $\phi:\, \sigma \ra \rho$ give rise to maps $\phi:\, \rmM_k(\sigma) \ra \rmM_k(\rho)$ on modular forms, in an analogous way.
\begin{proposition}
\label{prop:decomposition-of-modular-forms-spaces}
If $\rho = \rho_1 \oplus \rho_2$, then there is a canonical isomorphism of $\rmM_k(\rho)$ and $\rmM_k(\rho_1) \oplus \rmM_k(\rho_2)$.
\end{proposition}
\begin{proof}
Let $\iota_1 :\, \rho_1 \hra \rho$ and $\iota_2 :\, \rho_2 \hra \rho$ be the inclusions associated with the given decomposition of~$\rho$.  For $f_1 \in \rmM_k(\rho_1)$ and $f_2 \in \rmM_k(\rho_2)$, we have $(\iota_1 \circ f_1) + (\iota_2 \circ f_2) \in \rmM_k(\rho)$ as is easily verified.  The inverse is given by the corresponding projections
$\pi_1 :\, \rho \thra \rho_1$ and $\pi_2 :\, \rho \thra \rho_2$.
\end{proof}

We proceed to the definition of Peterssson scalar products.  Let $\rho$ and $\sigma$ be two representations of $\SL{2}(\ZZ)$ and fix an embedding $\iota :\, \bbone \hra \rho \otimes \ov{\sigma}$, where $\ov{\sigma}$ is the complex conjugate of~$\sigma$. Since we are in the unitary setting, there is an identification of~$\ov{\sigma} = \sigma^\vee$ be means of the scalar product on~$V(\sigma)$. In case that $\rho = \sigma$, there is a canonical choice given by $\iota(1) = \sum_v v \otimes \ov{v}$, where $v$ runs through an orthonormal basis of $V(\rho)$. In general, we have $\Hom(\bbone, \rho \otimes \sigma^\vee) \cong \Hom(\sigma, \rho)$. For $f \in \rmS_k(\rho)$ and $g \in \rmS_k(\sigma)$, we define
\begin{gather}
\label{eq:def:scalar-product}
  \langle f, g \rangle_\iota
=
  \int_{\SL{2}(\ZZ) \backslash \HS}
  \big\langle f \otimes \ov{g},\, \iota(1) \big\rangle
  \; \frac{d\!x\,d\!y}{y^{2 - k}}
\,\text{.}
\end{gather}
As usually, we can extend the Petersson scalar product to the case $f \in \rmM_k(\rho)$ by applying Borcherds regularization, explained at the beginning of Section~6 of~\cite{borcherds-1998}. Cusp forms are, as seen when unfolding, orthogonal to Eisenstein series, defined in Section~\ref{sec:eisenstein-series}. Details on the unfolding of regularized integrals can be found in~\cite{bruinier-2002a}, on page~47ff.

\subsection{Induced representations}
\label{ssec:preliminaries:induced-representations}

Recall that we focus on finite dimensional representations whose kernel has finite index in $\SL{2}(\ZZ)$. The induced representation $\Ind_\Gamma^{\Gamma'}\, \rho$ attached to $\Gamma \subseteq \Gamma' \subseteq \SL{2}(\ZZ)$ and a representation $\rho$ of $\Gamma$ can thus be defined by
\begin{gather}
\label{eq:def:induction-of-representations}
  V\big( \Ind_\Gamma^{\Gamma'} \rho \big)
:=
  V(\rho) \otimes \CC[B]
\quad\tx{and}\quad
  \Ind\,\rho (\gamma') \big( v \otimes \frake_\beta \big)
=
  \big( \rho\big( I_\beta^{-1}(\gamma^{\prime\, -1}) \big) v \big) \frake_{\beta \gamma^{\prime\,-1}}
\end{gather}
for a fixed system~$B$ of representatives of $\Gamma \backslash \Gamma'$ containing the identity element, and the following cocycle~$I$. For $\beta \in B$ and $\gamma' \in \Gamma'$, we set $\beta \gamma' = I_\beta(\gamma') \ov{\beta \gamma'}$ where $\ov{\beta \gamma'} \in B$ and $I_\beta(\gamma') \in \Gamma$.  This indeed defines a cocycle, that is, we have $I_\beta(\gamma'_1 \gamma'_2) = I_\beta(\gamma'_1) \, I_{\beta \gamma'_1}(\gamma'_2)$. To ease notation, we extend it to all of $\Gamma'$ by setting $I_\beta = I_{\ov\beta}$. We write $\frake_\beta$, $\beta \in B$ for the canonical orthonormal basis of $\CC[B]$.

We define a map $\Ind$ on modular forms.  For any $k \in \ZZ$ and any Dirichlet character~$\chi$ mod~$N$, we set
\begin{gather}
\label{eq:def:induction-on-modular-forms}
  \Ind
:\,
  \rmM_k(\Gamma_0(N), \chi)
\lra
  \rmM_k\big( \Ind_{\Gamma_0(N)}\,\chi \big)
\text{,}\quad
  f
\lmto
  \sum_{\gamma \in \Gamma_0(N) \backslash \SL{2}(\ZZ)}
  \big(f|_{k}\, \gamma\big)\, \frake_\gamma
\,\text{,}
\end{gather}
where $\gamma$ runs trough a fixed system of representatives of $\Gamma_0(N) \backslash \SL{2}(\ZZ)$. The group and character of a function will always be fixed separately, so that we throughout write $\Ind(f)$ without referring to the group or character attached to~$f$.
\begin{proposition}
The map $\Ind$ in~\eqref{eq:def:induction-on-modular-forms} is an isomorphism.
\end{proposition}
\begin{proof}
It is clear that~$\Ind$ is injective.  Its inverse is given by $f \mto \langle f,\, \frake_{I_2} \rangle$ for $f \in \rmM_k(\Ind_{\Gamma_0(N)}\, \chi)$, where $I_2$ is the $2 \times 2$ identity matrix.
\end{proof}

\begin{proposition}
\label{prop:classical-scalar-product}
For classical modular forms $f \in \rmM_k(\chi)$ and $g \in \rmS_k(\chi)$ with $\chi$ a Dirichlet character modulo~$N$ and the associated vector valued forms $\Ind\,f, \Ind\,g \in \rmM_k(\Ind_{\Gamma_0(N)}\chi)$, we have
\begin{gather*}
  \langle f, g \rangle
=
  \langle \Ind\,f,\, \Ind\,g \rangle_\iota
\text{,}\quad\text{with}\quad
  \iota(1)
=
  \frac{1}{[\SL{2}(\ZZ) : \Gamma_0(N)]}
  \sum_{\gamma \in \Gamma_0(N) \backslash \SL{2}(\ZZ)} \frake_\gamma \otimes \frake_\gamma
\,\text{,}
\end{gather*}
where $[\SL{2}(\ZZ) : \Gamma_0(N)]$ is the index of $\Gamma_0(N)$ in $\SL{2}(\ZZ)$.
\end{proposition}
\begin{proof}
Employing the definition of $\langle \Ind\,f,\, \Ind\,g \rangle_\iota$, we obtain
\begin{multline*}
  \int_{\SL{2}(\ZZ) \backslash \HS}
  \Big\langle
  \Ind\,f \otimes \Ind\,\ov{g},\,
  \frac{\sum_\gamma \frake_\gamma \otimes \frake_\gamma}{[\SL{2}(\ZZ) : \Gamma_0(N)]}
  \Big\rangle\;
  \frac{d\!x d\!y}{y^{2-k}}
\\
=
  I_N^{-1}
  \int_{\SL{2}(\ZZ) \backslash \HS}
  \sum_\gamma
  \big( f \big|_k\, \gamma \big) \ov{\big( g \big|_k\, \gamma \big)}\;
  \frac{d\!x d\!y}{y^{2-k}}
=
  I_N^{-1}
  \int_{\Gamma_0(N) \backslash \HS} f \ov{g}\;
  \frac{d\!x d\!y}{y^{2-k}}
\tx{,}
\end{multline*}
where $I_N = [\SL{2}(\ZZ) : \Gamma_0(N)]$.
\end{proof}

\section{Vector Valued Hecke Operators}
\label{sec:hecke-operators}

The definition of vector valued Hecke operators involves two steps.  In Section~\ref{ssec:hecke-operators-on-representations}, we study Hecke operators on representations, and Section~\ref{ssec:hecke-operators-on-modular-forms} contains a treatment of Hecke operators on modular forms. We show in Section~\ref{ssec:hecke-operators-known-constructions} that classical Hecke operators for modular forms attached to $\Gamma_0(N)$ and a Dirichlet character $\chi$ can be reconstructed from the new definition.

\subsection{Hecke operators on representations}
\label{ssec:hecke-operators-on-representations}

For a positive integer~$M$, let $\Delta_M = \big\{ \left(\begin{smallmatrix} a & b \\ 0 & d \end{smallmatrix}\right) \,:\, ad = M, 0 \le b < d \big\}$ be a set of upper triangular matrices of determinant~$M$ that are inequivalent with respect to the action of $\SL{2}(\ZZ)$ from the left. Given any $2 \times 2$ matrix $m$ with integer coefficients and determinant~$M$, there is $\gamma' \in \SL{2}(\ZZ)$ such that $\gamma' \ov{m} = m$ for some matrix $\ov{m} \in \Delta_M$. We will use this overline-notation throughout.

We have a right action of $\SL{2}(\ZZ)$ on $\Delta_M$ defined by $(m, \gamma) \mto \ov{m \gamma}$ with $\gamma' \ov{m \gamma} = m \gamma$ for some $\gamma' \in \SL{2}(\ZZ)$. It is readily verified that the associated map $I_m(\gamma) = \gamma'$ is a cocyle, satisfying $I_{m}(\gamma_1 \gamma_2) = I_m(\gamma_1) I_{m \gamma_1}(\gamma_2)$.  Note that the subscript $m$ makes it impossible to confuse $I_m$ and $I_\gamma$ defined in Section~\ref{ssec:preliminaries:induced-representations}.

The cocycle property of $I_m(\gamma)$ guaranties that, given a representation~$\rho$, then $T_M\,\rho$ which is defined by
\begin{gather}
\label{eq:def-hecke-operator-on-reps}
  V(T_M\, \rho)
:=
  V(\rho) \otimes \CC\big[ \Delta_M \big]
\quad\text{and}\quad
  (T_M\, \rho)(\gamma)\, (v \otimes \frake_{m})
:=
  \rho\big( I_m^{-1}(\gamma^{-1}) \big) (v) \otimes \frake_{m \gamma^{-1}}
\end{gather}
is a also representation.  Indeed, we have
\begin{align*}
&
  (T_M\,\rho)(\gamma_1) \big( (T_M\,\rho) (\gamma_2) (v \otimes \frake_m) \big)
=
  (T_M\,\rho)(\gamma_1) \big( \rho\big( I_m^{-1}(\gamma_2^{-1}) \big) v \otimes \frake_{m \gamma_2^{-1}} \big)
\\
={}&
  \rho\big( I^{-1}_{m \gamma_2^{-1}}(\gamma_1^{-1}) \big)\,
    \rho\big( I^{-1}_m(\gamma_2^{-1}) \big) v
  \otimes \frake_{m \gamma_2^{-1} \gamma_1^{-1}}
=
  \rho\big( I_m(\gamma_2^{-1}) I_{m\gamma_2^{-1}}(\gamma_1^{-1}) \big)^{-1} v
  \otimes \frake_{m (\gamma_1 \gamma_2)^{-1}}
\\
={}&
  \rho\big( I_m^{-1}\big((\gamma_1 \gamma_2)^{-1} \big) \big) v
  \otimes \frake_{m (\gamma_1 \gamma_2)^{-1}}
=
  (T_M\,\rho)(\gamma_1 \gamma_2) (v \otimes \frake_m)
\,\text{.}
\end{align*}

\begin{definition}
\label{def:hecke-operator}
Let $M$ be a positive integer. Then the $M$-th vector valued Hecke operator on representations is the assigment
\begin{gather*}
  T_M \,:\, \rho \lmto T_M\, \rho
\txt{,}
\end{gather*}
where $T_M\, \rho$ is the representation defined in~\eqref{eq:def-hecke-operator-on-reps}.
\end{definition}
\begin{remark}
\label{rm:hecke-operator}
For the purpose of reference, we note that $(T_M\,\rho)(\gamma^{-1})$ and $(T_M\,\rho)^\vee(\gamma)$ act by
\begin{gather*}
  (T_M\, \rho)(\gamma^{-1})\, (v \otimes \frake_{m})
:=
  \rho(I_m^{-1}(\gamma)) (v) \otimes \frake_{m \gamma}
\,\text{,}
\qquad
  (T_M\, \rho)^\vee(\gamma)\, (v \otimes \frake_{m})
:=
  \rho^\vee(I_m(\gamma)) (v) \otimes \frake_{m \gamma}
\,\text{.}
\end{gather*}
\end{remark}

The next two technical lemmas will be used frequently without further mentioning them.
\begin{lemma}
If $\rho$ has finite index kernel, then $T_M\,\rho$ has finite kernel, too.  If $\ker\,\rho$ is a congruence subgroup, then so is $\ker\,(T_M\, \rho)$.
\end{lemma}
\begin{proof}
A necessary condition for $\gamma \in \ker\, T_M\,\rho$ is $\ov{m \gamma^{-1}} = m$ for all $m \in \Delta_M$. For our purpose, it suffices to observe that the principal congruence subgroup $\Gamma(M)$ acts trivially on $\Delta_M$. Given $\gamma \in \Gamma(M)$, we have $\gamma \in \ker\, T_M\,\rho$ if $I_m^{-1}(\gamma^{-1}) = m^{-1} \gamma m \in \ker(\rho)$ for all~$m$. In other words, we have
\begin{gather*}
  \ker \big( T_M\,\rho \big)
\subseteq
   \Gamma(M) \cap \bigcap_{m \in \Delta_M} m (\ker\, \rho) m^{-1}
\,\tx{.}
\end{gather*}
If $\Gamma(N) \subseteq \ker\,\rho$ for some~$N$, we readily verify that $\Gamma(MN) \subseteq \ker\big( T_M\, \rho \big)$, completing the proof.
\end{proof}

\begin{lemma}
\label{la:hecke-operator-preserves-unitarity}
Assume that $\rho$ is unitary with respect to the scalar product $\langle \cdot\,,\, \cdot\, \rangle_\rho$ on $V(\rho)$.  Then $T_M\, \rho$ is unitary with respect to
\begin{gather*}
  \big\langle v \otimes \frake_{m_v},\, w \otimes \frake_{m_w} \big\rangle
=
  \begin{cases}
    \langle v,\, w \rangle_\rho
  \text{,}
  &
    \text{if $m_v = m_w$;}
  \\
    0
  \text{,}
  &
    \text{otherwise.}
  \end{cases}
\end{gather*}
\end{lemma}
\begin{proof}
This follows directly, since $\SL{2}(\ZZ)$ acts on the $\frake_m$ by permutations.
\end{proof}

The following two homomorphisms involving vector valued Hecke operators on representations are central to our further considerations. Their analogues for attached spaces of modular forms is stated in Theorem~\ref{thm:inclusion-of-hecke-operator-tensor-product-on-modular-forms}. Note that to ease discussion from now on we will identify a representation with its representation space. In the next statement and throughout the paper, inclusions are denoted by arrows~$\lhra$ and surjections are denoted by arrows~$\thra$. 
\begin{proposition}
\label{prop:hecke-operator-comonoidal-on-reps}
For every positive integer $M$, for every two finite dimensional $\SL{2}(\ZZ)$\nbd representations $\rho$ and $\sigma$, and for every homomorphism $\phi:\, \rho \ra \sigma$ between them, the following maps are homomorphisms of $\SL{2}(\ZZ)$\nbd representations:
\begin{alignat}{2}
\label{eq:prop:hecke-operator-comonoidal-on-reps:TMphi}
  T_M\,\phi :\;
  T_M\, \rho &\lra T_M\, \sigma
\tx{,}\qquad&
  v \otimes \frake_m &\lmto \phi(v) \otimes \frake_m
\,\tx{;}
\\[.3em]
\label{eq:prop:hecke-operator-comonoidal-on-reps:comonoidal-inclusion}
  \bbone &\lhra T_M\,\bbone
\tx{,}\qquad&
  c &\lmto c\, \sum_{m \in \Delta_M} \frake_m
\,\tx{;}
\\
\label{eq:prop:hecke-operator-comonoidal-on-reps:comonoidal-coherence}
  (T_M\, \rho) \otimes (T_M\, \sigma)
&\lthra
  T_M\, (\rho \otimes \sigma)
\tx{,}\qquad&
  (v \otimes \frake_m) \otimes (w \otimes \frake_{m'})
&\lmto
  \begin{cases}
    (v \otimes w) \otimes \frake_m
  \tx{,}
  &
    \tx{if $m = m'$;}
  \\
    0
  \tx{,}
  &
    \tx{otherwise.}
  \end{cases}
\end{alignat}
If $\psi :\, \sigma \ra \varpi$ is a further homomorphism of $\SL{2}(\ZZ)$\nbd representations, then $T_M\, (\psi \circ \phi) = (T_M\, \psi) \circ (T_M\, \phi)$.
\end{proposition}
\begin{proof}
The equality
\begin{gather*}
  (T_M\,\rho)(\gamma)\, (T_M\,\phi) (v \otimes \frake_m)
=
  \sigma\big( I_m^{-1}(\gamma^{-1}) \big) \big(\phi(v) \otimes \frake_{m \gamma^{-1}} \big)
=
  \phi\big( \rho\big(I_m^{-1}(\gamma^{-1}) \big) v \big) \otimes \frake_{m \gamma^{-1}}
\,\text{,}
\end{gather*}
which yields $(T_M\,\phi)\, (T_M\,\rho)(\gamma) (v \otimes \frake_m)$, establishes that $T_M\, \phi$ is a homomorphism. Compatibility with composition of homomorphisms is readily verified.

To show that~\eqref{eq:prop:hecke-operator-comonoidal-on-reps:comonoidal-inclusion} defines a homomorphism, it suffices to observe that $\SL{2}(\ZZ)$ acts on $\Delta_M$ by permutations. In particular, the action on $\sum_m \frake_m$ is trivial.

For the same reason, $m \gamma^{-1} = m' \gamma^{-1}$ implies that $m = m'$. In particular, for $m \ne m'$ we have
\begin{gather*}
  \big( T_M\, \rho \otimes T_M\, \sigma \big)(\gamma)\; 
  \big( (v \otimes \frake_m) \otimes (w \otimes \frake_{m'}) \big)
=
  \big( I_m^{-1}(\gamma^{-1})v \otimes \frake_{m \gamma^{-1}} \big)
  \otimes
  \big( I_{m'}^{-1}(\gamma^{-1})w \otimes \frake_{m' \gamma^{-1}} \big)
\lmto
  0  
\end{gather*}
under the map~\eqref{eq:prop:hecke-operator-comonoidal-on-reps:comonoidal-coherence}. The following computation thus proves that~\eqref{eq:prop:hecke-operator-comonoidal-on-reps:comonoidal-coherence} is a homomorphism:
\begin{multline*}
  \big( T_M\, \rho \otimes T_M\, \sigma \big)(\gamma)\; 
  \big( (v \otimes \frake_m) \otimes (w \otimes \frake_m) \big)
=
  \big( \rho(I_m^{-1}(\gamma^{-1})) v \otimes \frake_{m \gamma^{-1}} \big)
  \otimes
  \big( \sigma( I_m^{-1}(\gamma^{-1})) w \otimes \frake_{m \gamma^{-1}} \big)
\\
\lmto{}
  \big( \rho(I_m^{-1}(\gamma^{-1})) v \otimes \sigma( I_m^{-1}(\gamma^{-1})) w \big)
  \otimes \frake_{m \gamma^{-1}}
=
  \big( T_M\, (\rho \otimes \sigma) \big)(\gamma)\;
  \big( (v \otimes w) \otimes \frake_m \big)
\tx{.}
\end{multline*}
\end{proof}

\subsection{Hecke operators on modular forms}
\label{ssec:hecke-operators-on-modular-forms}

For $m = \left(\begin{smallmatrix} a & b \\ 0 & d \end{smallmatrix}\right) \in \GL{2}(\RR)$ and $f :\, \HS \ra \CC$, we define
\begin{gather*}
  \big( f \big|_k m \big) (\tau)
=
  \big(\tfrac{a}{d}\big)^{\frac{k}{2}}\,
  f\big(\frac{a\tau + b}{d}\big)
\text{.}
\end{gather*}

\begin{definition}
\label{def:hecke-operator-on-modular-forms}
Fix a representation $\rho$ of $\SL{2}(\ZZ)$ with finite index kernel and a positive integer~$M$.  Given $f \in \rmM_k(\rho)$, we define $T_M\,f$ by
\begin{gather}
  \big( T_M\, f \big)(\tau)
=
  \sum_{m \in \Delta_M} \big( f \big|_k m \big) \otimes \frake_m
\,\text{.}
\end{gather}
\end{definition}
\begin{proposition}
If $f \in \rmM_k(\rho)$, then $T_M\,f \in \rmM_k(T_M\, \rho)$.
\end{proposition}
\begin{proof}
We have to check that
\begin{gather*}
  (T_M\, f) \big|_k \gamma
=
  (T_M\, \rho)(\gamma)\, f
\end{gather*}
for all $\gamma \in \SL{2}(\ZZ)$. Our proof is a direct computation using the definitions given in ~\eqref{eq:def-hecke-operator-on-reps} and~\eqref{def:hecke-operator-on-modular-forms}.
\begin{multline*}
  \big( T_M\,f \big) \big|_{k, T_M\,\rho}\,\gamma
=
  \sum_{m \in \Delta_M} \big( \big( f \big|_k m \big) \otimes \frake_m \big) \big|_{k, T_M\,\rho}\, \gamma
=
  \sum_{m \in \Delta_M} \rho(I_m^{-1}(\gamma))\, \big( f \big|_k m \gamma \big) \otimes \frake_{m \gamma}
\\
=
  \sum_{m \in \Delta_M} \rho(I_m^{-1}(\gamma))\, \big( f \big|_k I_m(\gamma) \ov{m \gamma} \big) \otimes \frake_{m \gamma}
=  
  \sum_{m \in \Delta_M} \big( f \big|_k \ov{m \gamma} \big) \otimes \frake_{m \gamma}
=
  T_M\, f
\tx{.}
\end{multline*}
In the third equality, we have used the fact that $m \gamma = I_m(\gamma) \ov{m \gamma}$, by definition of the cocycle~$I$. The fourth equation follows from $f \in \rmM_k(\rho)$. The last equality follows when replacing $m \gamma$ by~$m' \in \Delta_M$.
\end{proof}

\begin{theorem}
\label{thm:inclusion-of-hecke-operator-tensor-product-on-modular-forms}
Fix a positive integer~$M$ and two representations $\rho$ and $\sigma$ of $\SL{2}(\ZZ)$. Write $\pi_{M,\rho,\sigma}$ for the homomorphism of representations defined in~\eqref{eq:prop:hecke-operator-comonoidal-on-reps:comonoidal-coherence}. For any weight $k \in \ZZ$, it gives rise to a linear map of modular forms by composition:
\begin{gather}
  \rmM_{k}\big( T_M\, \rho \otimes T_M\, \sigma \big)
\lra
  \rmM_{k}\big( T_M\, (\rho \otimes \sigma) \big)
\tx{,}\quad
  f \lmto \pi_{M,\rho,\sigma} \,\circ\, f
\tx{.}
\end{gather}
This map intertwines the tensor product of modular forms and the vector valued Hecke operator. That is, we have
\begin{gather}
  \pi_{M,\rho,\sigma} \,\circ\,
  \big( (T_M\, f) \otimes (T_M\, g) \big)
=
  T_M\, (f \otimes g)
\end{gather}
for any two modular forms $f \in \rmM_k(\rho)$ and $g \in \rmM_l(\sigma)$.
\end{theorem}
\begin{remark}
Compare Theorem~\ref{thm:inclusion-of-hecke-operator-tensor-product-on-modular-forms} to the fact that $(f \big|_k\, T_M) (g \big|_l\, T_M) \ne fg \big|_{k+l} T_M$ for general $f \in \rmM_k(\Gamma)$ and $g \in \rmM_l(\Gamma)$ and finite index subgroups $\Gamma \subseteq \SL{2}(\ZZ)$, where $\big|\, T_M$ is the classical Hecke operator.
\end{remark}
\begin{proof}[{Proof of Theorem~\ref{thm:inclusion-of-hecke-operator-tensor-product-on-modular-forms}}]
The first part follows immediately from the fact that $\pi_{M, \rho, \sigma}$ is a homomorphism. The second part is quickly verified:
\begin{gather*}
  \pi_{M,\rho,\sigma}
  \Big(
  \sum_{m, m' \in \Delta_M}
  \big( f\big|_k\, m \otimes \frake_m \big)
  \otimes
  \big( g\big|_l\, m' \otimes \frake_{m'} \big)
  \Big)
=
  \sum_{m \in \Delta_M}
  \big( f\big|_k\, m \big)
  \otimes
  \big( g\big|_l\, m \big)
  \,\otimes\, \frake_m 
\tx{.}
\end{gather*}
\end{proof}

\subsubsection{Hecke operators from induced representations}

We have decided to define Hecke operators in a way that resembles the classical construction most. Instead however, we could equivalently define $T_M\,\rho$ as a suitable $\SL{2}(\ZZ)$\nbd invariant subspace of an induced representation. We discuss this idea in a bit more detail.

Given a representation~$\rho$ of $\SL{2}(\ZZ)$ let
\begin{gather*}
  \wht\rho
=
  \Ind_{\SL{2}(\ZZ)}^{\GL{2}^+(\QQ)}\, \rho
\tx{,}
\end{gather*}
where $\GL{2}(\QQ)^+ \subset \GL{2}(\QQ)$ is the subgroup of matrices with positive determinant. Since $\GL{2}(\QQ)^+$ is discrete, a definition of the induction analogous with the one in Section~\ref{ssec:preliminaries:induced-representations} works in this case, ignoring slight topological complications. Fixing a set of representatives of $\SL{2}(\ZZ) \backslash \GL{2}(\QQ)^+$, set $\gamma \delta = \hat{I}_\gamma(\delta) \ov{\gamma \delta}$ for $\gamma$ and $\ov{\gamma \delta}$ one of these representative and $\hat{I}_\gamma(\delta) \in \SL{2}(\ZZ)$. Clearly, $\hat{I}_\gamma(\delta)$ is a cocyle.

For simplicity we assume that the chosen set of representatives of $\SL{2}(\ZZ) \backslash \GL{2}(\QQ)^+$ comprises all $\Delta_M$. We claim that
\begin{gather}
\label{eq:hecke-operator-as-induced-representation}
  \iota_{\Ind, M} :\,
  T_M\, \rho \lhra \wht\rho\,,\quad
  v \otimes \frake_m \lmto v \otimes \frake_m
\end{gather}
is a well-defined inclusion of $\SL{2}(\ZZ)$\nbd representations. In particular, $\bigoplus_M T_M\,\rho \hra \wht\rho$. Note that there is an associated projection $\pi_{\Ind, M}$, that maps $\frake_\gamma$ with $\gamma \in \SL{2}(\ZZ) \backslash \GL{2}(\QQ)^+$ to either $\frake_\gamma$ if $\gamma \in \Delta_M$ or to zero, otherwise.

It is clear that \eqref{eq:hecke-operator-as-induced-representation} is well-defined, because of our assumption that $\Delta_M$ is part of the representatives of $\SL{2}(\ZZ) \backslash \GL{2}(\QQ)^+$. Equivariance with respect to $\SL{2}(\ZZ)$ is almost part of the definition of the cocycles. Indeed, for $m \in \Delta_M$ and $\delta \in \SL{2}(\ZZ)$, we have $\hat{I}_m(\delta) \ov{m \delta} = m \delta = I_m(\delta) \ov{m \delta}$.

Next, we compare the Hecke operator on modular forms and the induction of modular forms. Given $f \in \rmM_k(\rho)$, we set $\wht\Ind(f) = \sum_\gamma \big( f |_k \gamma \big) \otimes \frake_\gamma$, where the sum runs over $\SL{2}(\ZZ) \backslash \GL{2}(\QQ)^+$. We claim that $\pi_{\Ind, M}$ intertwines $T_M$ and $\wht\Ind$. That is, we have
\begin{gather*}
  \pi_{\Ind, M} \big( \wht\Ind(f) \big)
=
  T_M\,f
\end{gather*}
for all $f \in \rmM_k(\rho)$. The proof is straightforward and again makes use of the fact that the cocycles $I_m(\delta)$ and $\hat{I}_m(\delta)$ are essentially the same.

\subsubsection{Adjunction}

In~\eqref{eq:def:scalar-product}, we have already introduced a scalar product on all spaces of vector valued modular forms that we consider. In the context of vector valued Hecke operator, we are naturally led to ask for a corresponding adjunction formula. To describe it, let $m^\# = \left(\begin{smallmatrix} d & -b \\ -c & a \end{smallmatrix}\right)$ be the adjoint of a $2 \times 2$ matrix~$m$. Set for a representation $\rho$ and a positive integer~$M$
\begin{gather}
  \iota_{\rm adj}
:\,
  \rho \lhra T_M\, T_M\, \rho
\,\tx{,}\quad
  v \lmto \sum_{m \in \Delta_M} v \otimes \frake_m \otimes \frake_{m^\#}
\,\tx{.}
\end{gather}
A corresponding projection is defined by
\begin{gather}
  \pi_{\rm adj}
:\,
  T_M\, T_M\, \rho \lthra \rho
\,\tx{,}\quad
  \sum_{m,m' \in \Delta_M} v_{m,m'} \otimes \frake_m \otimes \frake_{m'}
  \lmto
  \sum_{m \in \Delta_M} v_{m, m^\#}
\,\tx{.}
\end{gather}
\begin{proposition}
\label{prop:adjoint-of-hecke-operator}
For every representation $\rho$ and every positive integer~$M$ the maps $\iota_{\rm adj}$ and $\pi_{\rm adj}$ are inclusions and projections of representations, respectively.

Given modular forms~$f \in \rmM_k(\rho)$ and~$g \in \rmM_k(T_M\,\rho)$, we have
\begin{gather}
\label{eq:prop:adjoint-of-hecke-operator}
  \big\langle T_M\, f,\, g \big\rangle
=
  \big\langle f,\, \pi_{\rm adj}\, T_M\, g \big\rangle
=
  \big\langle \iota_{\rm adj}\, f,\, T_M\, g \big\rangle
\,\tx{.}
\end{gather}
\end{proposition}
\begin{proof}
It is clear that $\pi_{\rm adj}$ is adjoint to $\iota_{\rm adj}$ so that is suffices to show that the latter is a homomorphism of representations. Given $v \in V(\rho)$ and $\delta \in \SL{2}(\ZZ)$, we have to check that
\begin{gather*}
  (T_M\,T_M\,\rho)(\delta^{-1})\, \big( \iota_{\rm adj}(v) \big)
=
  \iota_{\rm adj}\big( \rho(\delta^{-1})(v) \big)
\tx{.}
\end{gather*}
The right hand side side equals
\begin{gather*}
  \sum_{m \in \Delta_M}
  \rho(\delta^{-1})(v)
  \otimes \frake_m \otimes \frake_{m^\#}
\,\tx{,}
\end{gather*}
while the left hand side is
\begin{multline*}
  (T_M\,T_M\,\rho)(\delta^{-1})
  \sum_{m \in \Delta_M}\!
  v \otimes \frake_m \otimes \frake_{m^\#}
=
  \sum_{m \in \Delta_M}\!
  \big( T_M\,\rho \big) \big( I^{-1}_{m^\#}(\delta) \big)\,
  \big( v \otimes \frake_m \big)
  \otimes \frake_{m^\# \delta}
\\
=
  \sum_{m \in \Delta_M}\!
  \rho\big( I^{-1}_m( I_{m^\#}(\delta) ) \big)(v)
  \otimes \frake_{m I_{m^\#}(\delta)}
  \otimes \frake_{m^\# \delta}
\,\tx{.}
\end{multline*}
To establish the first part of the proposition, we therefore have to show that
\begin{gather*}
  \ov{m^\# \delta}^\#
=
  \ov{ m I_{m^\#}(\delta) }
\quad\tx{and}\quad
  I_m( I_{m^\#}(\delta)
=
  \delta
\,\tx{.}
\end{gather*}
We show the second equality first. Observe that we have
\begin{gather}
\label{eq:prop:adjoint-of-hecke-operator:proof}
  m I_{m^\#}(\delta) \ov{m^\# \delta}
=
  m m^\# \delta = M \delta
\quad\tx{implying that}\qquad
  m I_{m^\#}(\delta)
=
  \delta\, M \ov{\big( m^\# \delta \big)}^{-1}
\in
  \Delta_M
\,\tx{.}
\end{gather}
In other words, $\delta$ satisfies the defining property of $I_m(I_{m^\#}(\delta))$. We deduce the first equality by evaluating the product
\begin{gather*}
  \ov{m I_{m^\#}(\delta)}\; \ov{m^\# \delta}
=
  I^{-1}_m( I_{m^\#}(\delta) ) m I_{m^\#(\delta)}\,
  I^{-1}_{m^\#}(\delta) m^\# \delta
=
  I^{-1}_m( I_{m^\#}(\delta) ) m m^\# \delta
=
  I^{-1}_m( I_{m^\#}(\delta) ) M \delta
=
  M
\tx{,}
\end{gather*}
where the last equality follows from rewriting~\eqref{eq:prop:adjoint-of-hecke-operator:proof}. This finishes our proof that $\iota_{\rm adj}$ is an inclusion of representations.

It remains to argue that~\eqref{eq:prop:adjoint-of-hecke-operator} is true, which says that
\begin{gather*}
  \big\langle T_M\, f,\, g \big\rangle
=
  \big\langle f,\, \pi_{\rm adj}\, T_M\, g \big\rangle
=
  \big\langle \iota_{\rm adj}\, f,\, T_M\, g \big\rangle
\tx{.}
\end{gather*}
Using that $\pi_{\rm adj}$ and $\iota_{\rm adj}$ are adjoint, it suffices to consider equality of~$\big\langle T_M\, f,\, g \big\rangle$ and $
\big\langle \iota_{\rm adj}\, f,\, T_M\, g \big\rangle$. In order to keep computations as explicit as possible let $v_i \in V(\rho)$ and $v_i^\vee$ be an orthonormal basis and its dual. Let $f_i$ and $g_{i,m}$ be the corresponding components of~$f$ and~$g$. The left hand side of~\eqref{eq:prop:adjoint-of-hecke-operator} equals
\begin{gather*}
  \int_{\SL{2}(\ZZ) \backslash \HS}\,
  \sum_{i,m} \big( f_i\big|_k\,m \big) \ov{\big( g_{i,m} \big)}\;
  \frac{d\!x d\!y}{y^{2-k}}
\,\tx{.}
\end{gather*}
The right hand side is
\begin{gather*}
  \int_{\SL{2}(\ZZ) \backslash \HS}\,
  \sum_{i,m,m'} f_i \delta(m^\#,m') \ov{\big( g_{i,m}\big|_k\,m' \big) \big)}\;
  \frac{d\!x d\!y}{y^{2-k}}
=
  \int_{\SL{2}(\ZZ) \backslash \HS}\,
  \sum_{i,m,m'} f_i \ov{\big( g_{i,m}\big|_k\,m^\# \big)}\;
  \frac{d\!x d\!y}{y^{2-k}}
\,\tx{,}
\end{gather*}
where $\delta$ denotes the Kronecker delta function. The remainder of the proof is the same as in the classical setting. We exploit invariance of the measure, applying $m \slash \sqrt{M} \in \SL{2}(\RR)$. This yields
\begin{multline*}
  \int_{\SL{2}(\ZZ) \backslash \HS}\,
  \sum_{i,m,m'} \big( f_i \big|_k\, m \slash \sqrt{M} \big)
                \ov{\big( g_{i,m}\big|_k\,m^\# m \slash \sqrt{M} \big) \big)}\;
  \frac{d\!x d\!y}{y^{2-k}}
\\
=
  \int_{\SL{2}(\ZZ) \backslash \HS}\,
  \sum_{i,m,m'} \big( f_i \big|_k\, m \slash \sqrt{M} \big)
                \ov{\big( g_{i,m}\big|_k\, \sqrt{M} I_2 \big) \big)}\;
  \frac{d\!x d\!y}{y^{2-k}}
=
  \int_{\SL{2}(\ZZ) \backslash \HS}\,
  \sum_{i,m,m'} \big( f_i \big|_k\, m  \big) \ov{\big( g_{i,m} \big)}\;
  \frac{d\!x d\!y}{y^{2-k}}
\,\tx{,}
\end{multline*}
which is the left hand side of~\eqref{eq:prop:adjoint-of-hecke-operator}, as stated above.
\end{proof}

\subsection{Connections with known constructions}
\label{ssec:hecke-operators-known-constructions}

In this section, we illustrate how to obtain classical constructions on scalar valued modular forms in terms of the operators $T_N$. For $0 < N \in \ZZ$, we set $\rho_N = \Ind_{\Gamma_0(N)}\,\bbone$, and for a Dirichlet character~$\chi$ mod~$N$, we set $\rho_\chi = \Ind_{\Gamma_0(N)}\,\chi$. In particular, if $\chi$ is the trivial character mod~$N$, then $\rho_\chi = \rho_N$.

As a first step, we identify $T_M\,\bbone$ with a sum of~$\rho_N$'s for suitable~$N$.
\begin{lemma}
\label{la:ind-isomorphic-TN1}
For any positive integer $M$, we have
\begin{gather}
  T_M\,\bbone
\cong
  \bigoplus_{a^2 \isdiv M} \Ind_{\Gamma_0(M \slash a^2)}\,\bbone
=
  \bigoplus_{a^2 \isdiv M} \rho_{M \slash a^2}
\,\tx{.}
\end{gather}
\end{lemma}
\begin{proof}
By definition, the representation $T_M\,\bbone$ is given by
\begin{gather}
  \gamma \frake_m = \frake_{m \gamma^{-1}}
\,\tx{,}\quad
  m \in \Delta_M
\,\tx{.}
\end{gather}
To decompose the permutation representation $T_M\,\bbone$ into irreducible components, it suffices to determine the orbits of $\SL{2}(\ZZ)$ acting on $\SL{2}(\ZZ) \Delta_M$ from the right. This is achieved by classical Hecke theory, which says that
\begin{gather*}
  \SL{2}(\ZZ) \Delta_M \SL{2}(\ZZ)
=
  \bigcup_{a;\, a^2 \isdiv M}\,
  \SL{2}(\ZZ)
  \left(\begin{smallmatrix}
    M \slash a & 0 \\ 0 & a
  \end{smallmatrix}\right)
  \SL{2}(\ZZ)
\end{gather*}
is a disjoint union of orbits. It thus suffices to determine the stabilizer of $\SL{2}(\ZZ) \left(\begin{smallmatrix} M \slash a & 0 \\ 0 & a \end{smallmatrix}\right)$ for each~$a$.

Given a positive integer $N = M \slash a^2$, the stabilizer of $\SL{2}(\ZZ)\diag(N,1)$ equals $\Gamma_0(N)$, where $\diag(N,1)$ is the matrix~$\left(\begin{smallmatrix} N & 0 \\ 0 & 1 \end{smallmatrix}\right)$. Indeed, a direct computation shows that the right action of $\Gamma_0(N)$ preserves it. On the other hand, it is known by Hecke theory that the $\SL{2}(\ZZ)$\nbd double coset generated by $\diag(N,1)$ has the same cardinality as the projective line over~$\ZZ \slash N \ZZ$. This shows that the index of the stabilizer in $\SL{2}(\ZZ)$ equals the index of $\Gamma_0(N)$, completing our argument.
\end{proof}

The remaining section builds up on the following purely representation theoretic computation. We set $m_M = \diag(M,1)$, $m'_M = \diag(1,M)$, and $m_{M,b} = \left(\begin{smallmatrix} M & b \\ 0 & M\end{smallmatrix}\right)$. Given $M$ and an exact divisor $N$ of $M$, we write $\gamma_{M,N} \in \SL{2}(\ZZ)$ for some fixed matrix that satisfies $\gamma_{M,N} \equiv \left(\begin{smallmatrix} 0 & -1 \\ 1 & 0 \end{smallmatrix}\right) \pmod{M}$ and $\gamma_{M,N} \equiv \left(\begin{smallmatrix} 1 & 0 \\ 0 & 1 \end{smallmatrix}\right) \pmod{M \slash N}$. To simplify notation, we assume that $\gamma_{M,N}$ is one of the fixed representatives of $\Gamma_0(N) \backslash \SL{2}(\ZZ)$. Further, set $e_M(x) = \exp(2\pi i\, x\slash M)$ for $x \in \CC$, and let $G(\epsilon, e_M(b)) = \sum_{a \pmod{M}} \epsilon(a) e_M(ab)$ be the Gauss sum attached to a Dirichlet character $\epsilon$ modulo~$M$. For $\chi$ a Dirichlet character mod~$N$, we define a (monoid) character on the set of matrices $m = \left(\begin{smallmatrix} a & b \\ c & d \end{smallmatrix}\right)$ with $N \isdiv c$ by $\chi(m) = \chi(d)$.
\begin{proposition}
\label{prop:components-of-hecke-induced}
Let $\chi$ be a Dirichlet character mod~$N$. For a positive integer~$M$, the maps
\begin{alignat}{3}
\label{eq:prop:components-of-hecke-induced:hecke}
&
  \iota_{\rm Hecke}:\,
&
  \rho_\chi &\lra T_M\,\rho_\chi
\,\tx{,}\quad
&
  \frake_\gamma
&\lmto\,
  M^{\frac{k}{2}-1}\hspace{-.4em}
  \sum_{m \in \Delta_M}
  \chi(m) \ov\chi\big( I_{I_2}(I_m(\gamma)) \big)\,
  \frake_{I_m(\gamma)} \otimes \frake_{m\gamma}
\,\tx{,}
\\[.2em]
\label{eq:prop:components-of-hecke-induced:al}
&
  \iota_{\rm AL}:\,
&
  \rho_{\chi'} &\lra T_M\,\rho_\chi
\,\tx{,}\quad
&
  \frake_\gamma
&\lmto\,
  \ov\chi\big( I_{I_2}(\gamma_{M,N} I_{m_M}(\gamma)) \big)\,
  \frake_{\gamma_{M,N} I_{m_M}(\gamma)} \otimes \frake_{m_M \gamma}
\,\tx{,}
\\[.5em]
\label{eq:prop:components-of-hecke-induced:oldform}
&
  \iota_{\rm old}:\,
&
  \rho_{\chi'} &\lra T_M\,\rho_\chi
\,\tx{,}\quad
&
  \frake_{\gamma}
&\lmto\,
  M^{-\frac{k}{2}}
  \ov\chi\big( I_{I_2}(I_{m_M}(\gamma)) \big)\,
  \frake_{I_{m_M}(\gamma)} \otimes \frake_{m_M \gamma}
\,\tx{,}
\\[.5em]
\label{eq:prop:components-of-hecke-induced:principal-congruence}
&
  \iota_{\Gamma(N)}:\,
&
  \rho_{\Gamma(N)} &\lra T_N\,\rho_{\Gamma_1(N^2,N)}
\,\tx{,}\quad
&
  \frake_{\gamma}
&\lmto\,
  M^{\frac{k}{2}}\,
  \frake_{I_{m'_N}(\gamma)} \otimes \frake_{m'_N \gamma}
\,\tx{,}
\\[.53em]
\label{eq:prop:components-of-hecke-induced:twist}
&
  \iota_{\rm twist}:\,
&
  \rho_{\chi'} &\lra T_{M^2}\,\rho_\chi
\,\tx{,}\quad
&
  \frake_{\gamma}
&\lmto\,
  M^{-1}\hspace{-.7em}
  \sum_{b \pmod{M}}\hspace{-.5em}
  G(\epsilon, e_M(-b))
  \ov\chi\big( I_{I_2}(I_{m_{M,b}}(\gamma)) \big)\,
  \frake_{I_{m_{M,b}}(\gamma)} \otimes \frake_{m_{M,b} \gamma}
\,\tx{,}
\\[.2em]
\label{eq:prop:components-of-hecke-induced:id}
&
  \iota_{\rm id}:\,
&
  \rho_{\chi} &\lra T_{M^2}\,\rho_\chi
\,\tx{,}\quad
&
  \frake_\gamma
&\lmto\,
  \frake_\gamma \otimes \frake_{\diag(M,M)}
\end{alignat}
are inclusions under the following circumstances:
\begin{enumerateroman}
\item
The map~\eqref{eq:prop:components-of-hecke-induced:hecke} is an inclusion if~$M$ is coprime to~$N$.

\item
The map~\eqref{eq:prop:components-of-hecke-induced:al} is an inclusion if~$M | N$ and $M$ is coprime to $N \slash M$, and $\chi'$ is the Dirichlet character mod~$N$ that equals $\chi$ mod~$N \slash M$ and that equals $\ov{\chi}$ mod $M$.

\item
The map~\eqref{eq:prop:components-of-hecke-induced:oldform} is an inclusion for $\chi'$ the mod~$MN$ Dirichlet character defined by~$\chi$.

\item
The map~\eqref{eq:prop:components-of-hecke-induced:principal-congruence} is always an inclusion.

\item
The map~\eqref{eq:prop:components-of-hecke-induced:twist} is an inclusion if $\epsilon$ is a Dirichlet character mod~$M$, and $\chi'$ the mod~$N M^2$ Dirichlet character defined by~$\chi \epsilon^2$.

\item
The map~\eqref{eq:prop:components-of-hecke-induced:id} is always an inclusion.
\end{enumerateroman}
\end{proposition}
\begin{remark}
\label{rm:components-of-hecke-induced:projections}
Using the scalar product on induced representations and their images under vector-valued Hecke operators (cf.~Lemma~\ref{la:hecke-operator-preserves-unitarity}), we obtain from each inclusion defined in the previous proposition a corresponding projection. More specifically, we define a projection $\pi$ associated to an inclusion by requiring that $\langle \pi(v), w \rangle = \langle v, \iota(w) \rangle$ for all vectors~$v$ and~$w$ in the domain of~$\pi$ and $\iota$, respectively. We denote the resulting projections by
\begin{gather}
\begin{alignedat}{6}
&
  \pi_{\rm Hecke}
:\,&
  T_M\,\rho_\chi &\lra \rho_\chi
\,\tx{,}\quad
&&
  \pi_{\rm AL}
:\,&
  T_M\,\rho_\chi &\lra \rho_\chi
\,\tx{,}\quad
&&
  \pi_{\rm old}
:\,&
  T_M\,\rho_\chi &\lra \rho_{\chi'}
\,\tx{,}\quad
\\
&
  \pi_{\Gamma(N)}
:\,&
  T_N\,\rho_{\Gamma_1(N^2,N)} &\lra \rho_{\Gamma(N)}
\,\tx{,}\quad
&&
  \pi_{\rm twist}
:\,&
  T_M\,\rho_\chi &\lra \rho_{\chi'}
\,\tx{,}\quad
&&
  \pi_{\rm id}
:\,&
  T_{M^2}\,\rho_\chi &\lra \rho_{\chi}
\,\tx{.}
\end{alignedat}
\end{gather}
\end{remark}

\begin{proof}[Proof of Proposition~\ref{prop:components-of-hecke-induced}]
There are at least two possible proofs of Proposition~\ref{prop:components-of-hecke-induced}. We will prove the first two cases by direct computation to illustrate various cocycle relations that enter the calculations. Then we give a more general, representation theoretic argument. Given $\gamma \in \Gamma_0(N) \backslash \SL{2}(\ZZ)$ and $\delta \in \SL{2}(\ZZ)$, we will check by computing the left and right hand side separately that
\begin{gather*}
  \iota\big( \rho(\delta^{-1})\, \frake_\gamma \big) 
=
  \rho'(\delta^{-1})\, \iota(\frake_\gamma)
\,\tx{,}
\end{gather*}
where $\rho$ is the repesentation on the domain of~\eqref{eq:prop:components-of-hecke-induced:hecke}, \eqref{eq:prop:components-of-hecke-induced:al}, etc.,\ $\rho'$ is the corresponding representation on the codomain, and $\iota$ is the inclusion of $\rho$ into $\rho'$.

We verify that~\eqref{eq:prop:components-of-hecke-induced:hecke} is an inclusion if $M$ and $N$ are coprime. Applying the definition of induced representations~$\rho_\chi$ and of $\iota_{\rm Hecke}$, we find that
\begin{align*}
&
  M^{1-\frac{k}{2}}\,
  \iota_{\rm Hecke} \big( \rho_\chi(\delta^{-1}) \frake_\gamma \big)
=
  M^{1-\frac{k}{2}}\,
  \iota_{\rm Hecke} \Big( \chi\big( I^{-1}_\gamma(\delta) \big)\, \frake_{\gamma\delta} \Big)
\\
={}&
  \chi\big( I^{-1}_\gamma(\delta) \big)
  \sum_{m \in \Delta_M}
  \chi(m) \ov\chi\big( I_{I_2}(I_m(\ov{\gamma \delta})) \big)\,
  \frake_{I_m(\ov{\gamma \delta})} \otimes \frake_{m \ov{\gamma \delta}}
\\
={}&
  \sum_{m \in \Delta_M}
  \chi(m) \ov\chi\big( I_\gamma(\delta) I_{I_2}^{-1}(I_m(\ov{\gamma \delta})) \big)\,
  \frake_{I_m(\ov{\gamma \delta})} \otimes \frake_{m \ov{\gamma \delta}}
\\
={}&
  \sum_{m' \in \Delta_M}
  \chi(m') \ov\chi\big( I_{I_2}(I_{m' I_\gamma(\delta)}(\ov{\gamma \delta})) \big)\;
  \frake_{I_{m' I_\gamma(\delta)}(\ov{\gamma \delta})}
  \otimes
  \frake_{m' I_\gamma(\delta) \ov{\gamma \delta}}
\,\tx{.}
\end{align*}
To obtain the last equality, we have replaced $m$ with $m'  I_\gamma(\delta)$. 

On the other hand, we have
\begin{align*}
&
  M^{1 - \frac{k}{2}}\, (T_M\,\rho_\chi)(\delta^{-1}) \big( \iota_{\rm Hecke}(\frake_\gamma) \big)
=
  (T_M\,\rho_\chi)(\delta^{-1})
  \Big(
  \sum_{m \in \Delta_M}
  \chi(m)
  \ov\chi\big( m I_{I_2}(I_m(\ov\gamma)) \big)\,
  \frake_{I_m(\ov\gamma)} \otimes \frake_{m \ov\gamma}
  \Big)
\\
={}&
  \sum_{m \in \Delta_M}\!\!
  \chi(m)
  \ov\chi\big( m I_{I_2}(I_m(\ov\gamma)) \big)\,
  \big( \rho_\chi \big( I_{m \ov\gamma}^{-1}(\delta) \big) \frake_{I_m(\ov\gamma)} \big)
  \otimes \frake_{m \ov\gamma \delta}
\\
={}&
  \sum_{m \in \Delta_M}\!\!
  \chi(m)
  \ov\chi\big( m I_{I_2}(I_m(\ov\gamma)) \big)\,
  \ov\chi\big( I_{I_m(\ov\gamma)}( I_{m \ov\gamma}(\delta) ) \big)\,
  \frake_{I_m(\ov\gamma) I_{m \ov\gamma}(\delta)}
  \otimes
  \frake_{m \ov\gamma \delta}
\\
={}&
  \sum_{m \in \Delta_M}
  \chi(m)
  \ov\chi\big( m I_{I_2}(I_m( \ov{\gamma} \delta)) \big)\,
  \frake_{I_m( \ov\gamma \delta)} \otimes \frake_{m \ov\gamma \delta}
\,\tx{.}
\end{align*}
To obtain the last equality we apply the cocycle relation to the subscript of the first tensor component, and the relation
\begin{gather*}
  I_{I_2}\big( I_m(\ov\gamma) \big)\, I_{I_m(\ov\gamma)} \big( I_{m \ov\gamma} (\delta) \big)
=
  I_{I_2}\big( I_m(\ov\gamma) I_{m\ov\gamma}(\delta) \big)
=
  I_{I_2}\big( I_m(\ov\gamma \delta) \big)
\end{gather*}
to the argument of the character.

To finish the case of~\eqref{eq:prop:components-of-hecke-induced:hecke}, we only need to show equality between the arguments of the character~$\chi$ and the subscripts of the tensor components. That is, we have to establish that
\begin{gather*}
  I_{m I_\gamma(\delta)}(\ov{\gamma \delta})
=
  I_m(\ov\gamma \delta)
\quad\tx{and}\quad
  m I_\gamma(\delta)\ov{\gamma\delta}
=
  m \ov{\gamma}\delta
\tx{.}
\end{gather*}
The second equality is a simple consequence of the defining formula of $I_\gamma(\delta)$. The first equality follows from the second one and from
\begin{gather*}
  I_m (\ov\gamma \delta)\,
  \ov{m I_\gamma(\delta) \ov{\gamma \delta}}
=
  I_m(\ov\gamma \delta)\, \ov{m \ov\gamma \delta}
=
  m \ov\gamma \delta
\tx{.}
\end{gather*}

To establish the case of~\eqref{eq:prop:components-of-hecke-induced:al}, consider the following equality, baring in mind that $\ov\gamma = \gamma$.
\begin{align*}
&
  (T_M \rho_\chi)(\delta^{-1})\big( \iota_{\rm AL}(\frake_\gamma) \big)
=
  (T_M \rho_\chi)(\delta^{-1})\Big(
  \ov\chi\big( I_{I_2}(\gamma_{M,N} I_{m_M}(\ov\gamma)) \big)\,
  \frake_{\gamma_{M,N} I_{m_M}(\ov\gamma)} \otimes \frake_{m_M \ov\gamma}
  \Big)
\\
={}&
  \ov\chi\big( I_{I_2}\big( \gamma_{M,N} I_{m_M}(\ov\gamma) \big) \big)\,
  \big( \rho_\chi\big( I^{-1}_{m_M \ov\gamma}(\delta) \big)\,
  \frake_{\gamma_{M,N} I_{m_M}(\ov\gamma)}
  \otimes
  \frake_{m_M \ov\gamma \delta}
\\
={}&
  \ov\chi\big( I_{I_2}\big( \gamma_{M,N} I_{m_M}(\ov\gamma) \big) \big)\,
  \ov\chi\big( I_{\gamma_{M,N} I_{m_M}(\ov\gamma)}\big( I_{m_M \ov\gamma}(\delta) \big) \big)
  \frake_{\gamma_{M,N} I_{m_M}(\ov\gamma) I_{m_M \ov\gamma}(\delta)}\,
  \otimes
  \frake_{m_M \ov\gamma \delta}
\\
={}&
  \ov\chi\big( I_{I_2}\big( \gamma_{M,N} I_{m_M}(\ov\gamma \delta) \big) \big)\,
  \frake_{\gamma_{M,N} I_{m_M}(\ov\gamma \delta)}
  \otimes
  \frake_{m_M \ov\gamma \delta}
\,\tx{.}
\end{align*}
The argument of the character simplifies, because
\begin{gather*}
  I_{I_2}\big( \gamma_{M,N} I_{m_M}(\ov\gamma) \big)
  I_{\gamma_{M,N} I_{m_M}(\ov\gamma)} \big( I_{m_M \ov\gamma}(\delta) \big)
=
  I_{I_2}\big( \gamma_{M,N} I_{m_M}(\ov\gamma) I_{m_M \ov\gamma}(\delta) \big)
=
  I_{I_2}\big( \gamma_{M,N} I_{m_M}(\ov\gamma \delta) \big)
\end{gather*}
by applying the cocycle relation twice.

Comparing that expression to
\begin{gather*}
  \iota_{\rm AL}\big( \rho_{\chi'}(\delta^{-1})\, \frake_{\gamma} \big)
=
  \iota_{\rm AL}\big( \ov\chi'\big( I_\gamma(\delta) \big)\, \frake_{\gamma \delta} \big)
=
  \ov\chi'\big( I_\gamma(\delta) \big)
  \ov\chi\big(I_{I_2} \big( \gamma_{M,N} I_{m_M}(\ov{\gamma\delta}) \big) \big)\,
  \frake_{\gamma_{M,N} I_{m_M}(\ov{\gamma\delta})} \otimes \frake_{m_M \ov{\gamma \delta}}
\end{gather*}
we need to show that
\begin{gather}
\label{eq:components-of-hecke-induced:al-character}
  \chi\Big(
  I_{I_2}\big( \gamma_{M,N} I_{m_M}(\ov\gamma \delta) \big)
  \Big)
=
  \chi'\big( I_\gamma(\delta) \big)\,
  \chi\Big(
  I_{I_2}\big( \gamma_{M,N} I_{m_M}(\ov{\gamma \delta}) \big)
  \Big)
\quad\tx{and}\quad
  m_M \ov\gamma \delta
=
  m_M \ov{\gamma \delta}
\tx{.}
\end{gather}
Since $\ov\gamma \delta = I_\gamma(\delta) \ov{\gamma \delta}$ with $I_\gamma(\delta) \in \Gamma_0(N)$, we immediately obtain the second equality: Indeed, by what is argued in the proof of Proposition~\ref{prop:components-of-hecke-induced}, $\Gamma_0(N)$ stabilizes $m_M$, since we have $M \isdiv N$.

It is more involved to illustrate the first auxiliary equality. The cocylce equality and $\ov{m I_\gamma(\delta)} = \ov{m}$ imply that
\begin{gather*}
  I_{m_M}(\ov\gamma \delta)
=
  I_{m_M I_\gamma(\delta)}(\ov{\gamma \delta})
=
  I_{m_M} \big( I_\gamma(\delta) \big)
  I_{m_M}(\ov{\gamma \delta})
\,\tx{.}
\end{gather*}
We therefore have to analyze
\begin{gather*}
  \chi\Big(
  I_{I_2}\big( \gamma_{M,N} I_{m_M}(\ov\gamma \delta) \big)
  \Big)
=
  \chi\Big(
  I_{I_2}\big( \gamma_{M,N}
               I_{m_M} \big( I_\gamma(\delta) \big)
               I_{m_M}(\ov{\gamma \delta})
  \big) \Big)
\tx{.}
\end{gather*}

Since we are interested only in a value of~$\chi$, which is a Dirichlet character mod~$N$, it suffices to consider its argument mod~$N$. Additionally, because $I_{I_2}$ is a cocycle for the action of $\SL{2}(\ZZ)$ on the cosets $\Gamma_0(N) \backslash \SL{2}(\ZZ)$, its value mod $N$ depends only on its argument mod~$N$. Therefore, all the following computations can be performed modulo~$N$. Since further $M$ and $N \slash M$ are coprime, it indeed suffices to compute mod~$M$ and mod~$N \slash M$ separately.

We first consider values mod $N \slash M$, which is easier since $\gamma_{M,N} \equiv I_2 \pmod{N \slash M}$. We have
\begin{gather*}
  I_{I_2}\big( \gamma_{M,N}
               I_{m_M} \big( I_\gamma(\delta) \big)
               I_{m_M}(\ov{\gamma \delta})
  \big)
\equiv
  I_{I_2}\big( I_{m_M} \big( I_\gamma(\delta) \big)
               I_{m_M}(\ov{\gamma \delta})
  \big)
\equiv
  I_{m_M} \big( I_\gamma(\delta) \big)
  I_{I_2}\big( I_{m_M}(\ov{\gamma \delta}) \big)
\tx{.}
\end{gather*}
The last congruence deserves further justification. However, we first finish the computation modulo~$N \slash M$. Namely, we have to compare $\chi\big( I_{m_M} \big( I_\gamma(\delta) \big) \big)$ and $\chi\big( I_\gamma(\delta) \big)$. We have already noticed that $I_\gamma(\delta)$ stabilizes $m_M$. Therefore $I_{m_M} \big( I_\gamma(\delta) \big) = m_M I_\gamma(\delta) m_M^{-1}$. The diagonal entries of $m_M I_\gamma(\delta) m_M^{-1}$ coincides with the ones of $I_\gamma(\delta)$. Since $\chi$ and $\chi'$ coincide mod~$N \slash M$, this finishes the mod~$N \slash M$ computations for~\eqref{eq:components-of-hecke-induced:al-character}.

Let us now argue that the last equality in the above equation holds. From the previously stated expression for~$I_{m_M} \big( I_\gamma(\delta) \big)$ we see that it lies in $\Gamma_0(N \slash M)$. Among other things, this implies that
\begin{gather*}
  \ov{ I_{m_M} \big( I_\gamma(\delta) \big) I_{m_M}(\ov{\gamma\delta}) }
=
  \ov{ I_{m_M}(\ov{\gamma\delta}) }
\tx{.}
\end{gather*}
By definition of the cocyles, we have
\begin{gather*}
  I_{I_2}\big(
  I_{m_M} \big( I_\gamma(\delta) \big)
  I_{m_M}( \ov{\gamma\delta} )
  \big)
  \ov{ I_{m_M}(\ov{\gamma\delta}) }
\equiv
  I_{I_2}\big(
  I_{m_M} \big( I_\gamma(\delta) \big)
  I_{m_M}( \ov{\gamma\delta} )
  \big)
  \ov{ I^{-1}_{m_M} \big( I_\gamma(\delta) \big)
       I_{m_M}(\ov{\gamma\delta}) }
\equiv
  I_{m_M} \big( I_\gamma(\delta) \big)
  I_{m_M}( \ov{\gamma\delta} )
\end{gather*}
and
\begin{gather*}
  I_{I_2}\big(
  I_{m_M}( \ov{\gamma\delta} )
  \big)
  \ov{ I_{m_M}(\ov{\gamma\delta}) }
\equiv
  I_{m_M}( \ov{\gamma\delta} )
\tx{.}
\end{gather*}
Combining both equalities, we obtain
\begin{gather*}
  I_{m_M} \big( I_\gamma(\delta) \big)
  I_{I_2}\big(
  I_{m_M} \big( I_\gamma(\delta) \big)
  I_{m_M}( \ov{\gamma\delta} )
  \big)
\equiv
  I_{I_2}\big(
  I_{m_M}( \ov{\gamma\delta} )
  \big)
\tx{,}
\end{gather*}
which yields precisely the relation that we have employed above. This finishes our computations mod~$N \slash M$.

Computations mod~$M$ depends on $\gamma_{M,N} \equiv \left(\begin{smallmatrix} 0 & -1 \\ 1 & 0 \end{smallmatrix}\right)$, $\gamma_{M,N} I_{m_M} \big( I_\gamma(\delta) \big) \in \Gamma_0(M)$, and on $\chi'$ being the inverse of $\chi$ mod~$M$. The rest of the considerations is the same as before. We leave details to the reader, and finish discussion of~\eqref{eq:prop:components-of-hecke-induced:al}.

The remaining cases~\eqref{eq:prop:components-of-hecke-induced:oldform}, \eqref{eq:prop:components-of-hecke-induced:principal-congruence}, \eqref{eq:prop:components-of-hecke-induced:twist}, and \eqref{eq:prop:components-of-hecke-induced:id} follow a similar pattern, but in particular~\eqref{eq:prop:components-of-hecke-induced:twist} would be very involved to establish. We therefore give an alternative proof, which is more conceptual, but allows for less insight into what transformations show up.

Consider the case~\eqref{eq:prop:components-of-hecke-induced:twist}, and fix $\chi$ and $\chi'$ as in the assumptions. For any $k \in \ZZ$ and any $f \in \rmM_k(\chi)$, we have, as in the proof of Proposition~\ref{prop:twists-of-modular-forms} that
\begin{gather*}
  \mathrm{twist}_\epsilon(f)
=
  M^{-1}
  \sum_b G(\ov\epsilon, e_M(b))
  f \big|_k\, m_{M,b}
\in
  \rmM_k(\chi')
\tx{.}
\end{gather*}
Comparing with the definition of induction of represenations in~\eqref{eq:def:induction-of-representations}, the space
\begin{gather*}
  \lspan \big\{
  \mathrm{twist}_\epsilon(f) \big|_k\, \gamma \,:\, \gamma \in \SL{2}(\ZZ)
  \big\}
\end{gather*}
with left representation $\gamma g = g \big|_k\, \gamma^{-1}$ is isomorphic to $\Ind\, \rho_{\ov\chi'}$, since $f \in \rmM_k(\chi')$. As a space of functions, it coincides with
\begin{gather*}
  \lspan \big\{
  \big( M^{-1}
  \sum_b G(\ov\epsilon, e_M(b))
  f \big|_k\, m_{M,b}
  \big) \big|_k\, \gamma \,:\, \gamma \in \SL{2}(\ZZ)
  \big\}
\tx{.}
\end{gather*}
The latter is a subspace of 
\begin{gather*}
  \lspan \big\{
  f \big|_k\, \gamma m \,:\, \gamma \in \SL{2}(\ZZ), m \in \Delta_{M^2}
  \big\}
\tx{,}
\end{gather*}
which is isomorphic to $T_{M^2}\, \rho_{\ov\chi}$ when considered as a space with left representation $\gamma g = g \big|_k\, \gamma^{-1}$. The embedding is given via the conjugate of~\eqref{eq:prop:components-of-hecke-induced:twist}, which proves the proposition.
\end{proof}

\subsubsection{The identity map}

For later use, we have to recover $f$ from $T_{M^2}\, f$.
\begin{lemma}
\label{la:identity-in-hecke-operator-square}
Let $\chi$ be a Dirichlet character mod~$N$.  Fix a positive integer $M$. The inclusion $\iota_{\rm id}$ and the corresponding projection~$\pi_{\rm id}$ intertwine the vector valued Hecke operator~$T_{M^2}$ and the identity map with induction from~$\Gamma_0(N)$.  For every $f \in \rmM_k(\chi)$ and $v \in V(\rho_\chi)$, we have
\begin{gather*}
  \big\langle \Ind(f),\, v \big\rangle
=
  \Big\langle \pi_{\rm id}\big( T_{M^2}\,\Ind(f) \big),\, v \Big\rangle
=
  \Big\langle T_{M^2}\,\Ind(f),\, \iota_{\rm id}(v) \Big\rangle
\tx{.}
\end{gather*}
\begin{gather*}
\end{gather*}
\end{lemma}
\begin{proof}
In light of Proposition~\ref{prop:components-of-hecke-induced} and the definition of $\pi_{\rm id}$ in terms of $\iota_{\rm id}$, it suffices to check that the left and right hand side agree for $v = \frake_{I_2}$. We have $\big\langle \Ind(f),\, \frake_{I_2} \big\rangle = f$ and
\begin{gather*}
  \big\langle T_{M^2}\,\Ind(f),\, \iota_{\rm id}(\frake_{I_2}) \big\rangle
=
  \big\langle T_{M^2}\,\Ind(f),\, \frake_{I_2} \otimes \frake_{\diag(M,M)} \big\rangle
=
  f \big|_k\, \diag(M,M)
=
  f
\tx{.}
\end{gather*}
\end{proof}

\subsubsection{Classical Hecke operators}

One obviously wants to recover classical Hecke operators from Definition~\ref{def:hecke-operator-on-modular-forms}.  Given $f \in \rmM_k(\chi)$, let $f \big| T_M$ be its image under the classical Hecke operator:
\begin{gather*}
  \big( f\big|_{k, \chi} T_M \big) (\tau)
=
  M^{k-1}
  \sum_{\substack{d \isdiv M \\ 0 \le b < d}}
  d^{-k}\, \ov{\chi}(d)\, f\Big( \frac{Md^{-1}\, \tau + b}{d} \Big)
\text{.}  
\end{gather*}
\begin{proposition}
\label{prop:classical-hecke-operator}
Let $\chi$ be a Dirichlet character mod~$N$.  Fix a positive integer $M$ that is coprime to $N$.  The inclusion $\iota_{\rm Hecke}$ and the corresponding projection~$\pi_{\rm Hecke}$ intertwine the vector valued Hecke operator and the classical Hecke operator with the induction map. For every $f \in \rmM_k(\chi)$ and $v \in V(\rho_\chi)$, we have
\begin{gather*}
  \big\langle \Ind\big( f \big|_{k,\chi}\, T_M \big),\, v \big\rangle
=
  \Big\langle \pi_{\rm Hecke}\big( T_M\,\Ind(f) \big),\, v \Big\rangle
=
  \Big\langle T_M\,\Ind(f),\, \iota_{\rm Hecke}(v) \Big\rangle
\text{.}
\end{gather*}
\end{proposition}
\begin{proof}
As in the proof of Lemma~\ref{la:identity-in-hecke-operator-square}, we can focus on the case of $v = \frake_{I_2}$. The left hand side then equals $f \big|_{k,\chi}\, T_M$, and the right hand side is
\begin{multline*}
  \big\langle T_M\,\Ind(f),\, \iota_{\rm Hecke}(\frake_{I_2}) \big\rangle
=
  \Big\langle
  \sum_{m, \gamma} \big( f \big|_k\, \gamma m \big)\, \frake_\gamma \otimes \frake_m,\,
  M^{\frac{k}{2}-1} \sum_{m \in \Delta_M} \chi\big( m \big)\, \frake_{I_2} \otimes \frake_{m}
  \Big\rangle
\\
=
  M^{\frac{k}{2}-1} \sum_m \ov{\chi}(m)\, f \big|_k\, m
=
  \Big\langle \sum_\gamma \big( f \big|_{k,\chi}\, T_M \gamma \big)\, \frake_\gamma ,\, \frake_{I_2} \Big\rangle
\,\tx{.}
\end{multline*}
\end{proof}

\subsubsection{Atkin--Lehner involutions}

Atkin--Lehner involutions for level~$N$ modular forms, which are defined under the assumptions of the next Proposition~\ref{prop:classical-atkin-lehner-operator}, map $f \in \rmM_k(N)$ to
\begin{gather*}
  W_M(f)
=
  f \big|_{k}\, \gamma_{M,N} m_M
\text{.}
\end{gather*}
\begin{proposition}
\label{prop:classical-atkin-lehner-operator}
Let $\chi$ be a Dirichlet character mod~$N$.  Fix a positive integer $M \isdiv N$ such that $M$ and $N \slash M$ are coprime. Let $\chi'$ be the Dirichlet character mod~$N$ that equals $\chi$ mod~$N \slash M$ and that equals $\ov\chi$ mod~$M$. The inclusion $\iota_{\rm AL}$ and the corresponding projection~$\pi_{\rm AL}$ intertwine the Atkin Lehner operator and Hecke operator with the induction map.  For $f \in \rmM_k(\Gamma_0(N), \chi)$ and $v \in V(\rho_\chi)$, we have
\begin{gather*}
  \big\langle \Ind_{\Gamma_0(N)}\, W_M(f),\, v \big\rangle
=
  \Big\langle \pi_{\rm AL}\big( T_M\, \Ind(f) \big),\, v \Big\rangle
=
  \Big\langle T_M\, \Ind(f),\, \iota_{\rm AL}(v) \Big\rangle
\text{.}
\end{gather*}
\end{proposition}
\begin{proof}
In analogy with the proof of Proposition~\ref{prop:classical-hecke-operator} it suffices to note that for $v = \frake_{I_2}$ the right hand side equals
\begin{gather*}
  \big\langle
  T_M\, \Ind(f),\,
  \chi'(\gamma_{M,N})\, \frake_{\gamma_{M,N}} \otimes \frake_{m_M}
  \big\rangle
=
  \chi'\big( I_{I_2}(\gamma_{M,N}) \big)\,
  f \big|_k\, \gamma_{M,N} m_M
=
  f \big|_k\, \gamma_{M,N} m_M
\,\tx{.}
\end{gather*}
\end{proof}

\subsubsection{Oldforms}

Oldforms in the scalar valued setting are obtained as $f(M \tau)$ for a given modular form $f$.  We recover this construction using vector valued Hecke operators. For convenience let
\begin{gather}
\label{eq:def:rescaling-of-arguments}
  \big( {\rm sc}_{M \slash M'}\,f \big)(\tau)
=
  f\big( M \tau \slash M' \big)
\end{gather}
be the map that rescales the argument of a function. Clearly, $\big( {\rm sc}_M\, f \big) (\tau) = M^{-\frac{k}{2}}\, f \big|_k \left(\begin{smallmatrix} M & 0 \\ 0 & 1 \end{smallmatrix}\right)$ is the oldform construction.
\begin{proposition}
\label{prop:classical-oldforms}
Let $\chi$ be a Dirichlet character mod~$N$.  Fix a positive integer~$M$, and let $\chi_M$ be the mod~$MN$ Dirichlet character defined by~$\chi$. The inclusion $\iota_{\rm old}$ and the corresponding projection~$\pi_{\rm old}$ intertwine the oldform construction and the Hecke operator with the induction map.  For $f \in \rmM_k(\chi)$ and $v \in V(\rho_{\chi_M})$, we have
\begin{gather*}
  \big\langle \Ind\big( {\rm sc}_M f \big),\, v \big\rangle
=
  \Big\langle \pi_{\rm old}\big( T_M\, \Ind(f) \big),\, v \Big\rangle
=
  \Big\langle T_M\, \Ind(f),\, \iota_{\rm old}(v) \Big\rangle
\text{.}
\end{gather*}
\end{proposition}
\begin{proof}
As in the other cases, we inspect the right hand side for $v = \frake_{I_2}$.
\begin{gather*}
  \big\langle T_M\, \Ind(f),\, M^{-\frac{k}{2}}\, \frake_{I_2} \otimes \frake_{m_M} \big\rangle
=
  M^{-\frac{k}{2}}\, f \big|_k \frake_{m_M}
\,\tx{.}
\end{gather*}
\end{proof}

\subsubsection{Rescaling for principal congruence subgroups}

Like the oldform construction, it is very common to consider $f(N \tau)$ for $f$ a modular form for the principal congruence subgroup~$\Gamma(N)$. The resulting $f(N \tau)$ is a modular form for $\Gamma_1(N^2,N)$. For technical reasons, we consider the reverse map.
\begin{proposition}
\label{prop:classical-principal-congruence-subgroups}
Let $N$ be a positive integer. The inclusion~$\iota_{\Gamma(N)}$ and the projection~$\pi_{\Gamma(N)}$ corresponding to it intertwine the map ${\rm sc}_{1 \slash N}$ from $\rmM_k(\Gamma_1(N^2, N))$ to $\rmM_k(\Gamma(N))$. For $f \in \rmM_k(\Gamma_1(N^2, N))$ and $v \in V(\rho_{\Gamma(N)})$, we have
\begin{gather*}
  \big\langle \Ind\big( {\rm sc}_{1\slash N} f \big),\, v \big\rangle
=
  \Big\langle \pi_{\Gamma(N)}\big( T_N\, \Ind(f) \big) ,\, v \Big\rangle
=
  \Big\langle T_N\, \Ind(f),\, \iota_{\Gamma(N)}(v) \Big\rangle
\,\tx{.}
\end{gather*}
\end{proposition}
\begin{proof}
We inspect the right hand side for $v = \frake_{I_2}$.
\begin{gather*}
  \big\langle
  T_N\, \Ind(f),\,
  N^{\frac{k}{2}}\, \frake_{I_2} \otimes \frake_{m'_N}
  \big\rangle
=
  N^{\frac{k}{2}}\, f \big|_k \frake_{m'_N}
\end{gather*}
\end{proof}

\subsubsection{Twists of modular forms}

We show that also twists of modular forms by Dirichlet characters can be recovered from vector valued Hecke operators. For a modular form $f \in \rmM_k(\chi)$ and a Dirichlet character~$\epsilon$ mod~$M$, there is a twist $f_\epsilon = {\rm twist}_\epsilon(f) \in \rmM_k(\chi')$ with Fourier coefficients $c(f_\epsilon ;\, n) = \epsilon(n) c(f;\, n)$, where $\chi'$ is the mod~$N M^2$ Dirichlet character defined by $\chi\epsilon^2$.  One can reconstruct~$f_\epsilon$ from $T_{M^2}\,f$.  Note that the modular form~${\rm sc}_{1 \slash M} {\rm twist}_\epsilon\, f$ already appears as a component of $T_M\, \Ind(f)$, since it can be written by~\eqref{eq:prop:twists-of-modular-forms:proof} below as
\begin{multline*}
  \mathrm{sc}_{1 \slash M} f_\epsilon
=
  \mathrm{sc}_{1 \slash M} \big(
  M^{-1}
  \sum_{b \pmod{M}} G(\epsilon, e_M(-b))\, f \big|_k\, m_{M,b}
  \big)
\\
=
  M^{\frac{k}{2}-1}
  \sum_{b \pmod{M}} G(\epsilon, e_M(-b))\,
  f \big|_k\, m_{M,b} \left(\begin{smallmatrix} 1 & 0 \\ 0 & M \end{smallmatrix}\right)
=
  M^{\frac{k}{2}-1}
  \sum_{b \pmod{M}} G(\epsilon, e_M(-b))\,
  f \big|_k\, \left(\begin{smallmatrix} 1 & b \\ 0 & M \end{smallmatrix}\right)
\tx{.}
\end{multline*}
\begin{proposition}
\label{prop:twists-of-modular-forms}
Let $\chi$ be a Dirichlet character mod~$N$.  Fix another Dirichlet character $\epsilon$ mod~$M$, and let $\chi'$ be as above. The inclusion $\iota_{\rm twist}$ and the corresponding projection~$\pi_{\rm twist}$ intertwine ${\rm twist}_\epsilon$ and the Hecke operator with the induction map.  For $f \in \rmM_k(\chi)$ and $v \in V(\rho_\chi)$, we have
\begin{gather*}
  \big\langle \Ind\big( {\rm twist}_\epsilon f \big),\, v \big\rangle
=
  \Big\langle \pi_{\rm twist} \big( T_{M^2}\, \Ind(f) \big),\, v \Big\rangle
=
  \Big\langle T_{M^2}\, \Ind(f),\, \iota_{\rm twist} (v) \Big\rangle
\tx{.}
\end{gather*}
\end{proposition}
\begin{proof}
We have $q^n \big|_k\, m_{M,b} = e_M(nb) q^{n}$.  This allows us to write twists of a modular form $f \in \rmM_k(\chi)$ as
\begin{multline}
\label{eq:prop:twists-of-modular-forms:proof}
  f_\epsilon(\tau)
=
  \sum_n \epsilon(n) c(n) q^{n}
=
  \sum_{n_0 \pmod{M}} \epsilon(n_0)\,
  M^{-1}
  \sum_{b \pmod{M}} e_M(-n_0 b) f \big|_k\, m_{M,b}
\\
=
  M^{-1}
  \sum_{b \pmod{M}} G(\epsilon, e_M(-b))\, f \big|_k\, m_{M,b}
\,\tx{.}
\end{multline}
With this at hand, inspecting the right hand side of the proposition's statement yields a proof.
\end{proof}

\section{Vector Valued Eisenstein Series}
\label{sec:eisenstein-series}

Recall that we assume throughout that $\rho$ has finite index kernel.  Fix an even integer~$k > 2$ and $v \in V(\rho)$.  We define the stabilizer ${\rm Stab}(v)$ of $v$ as $\{ \gamma \in \SL{2}(\ZZ) \,:\, \rho(\gamma)v = v \big\}$. Write $\Gamma_\infty(v)$ for the intersection of $\Gamma_\infty$ and ${\rm Stab}(v)$. Observe that $\Gamma_\infty(v)$ has finite index in $\Gamma_\infty$, since $\rho$ has finite index kernel. The series 
\begin{gather}
\label{eq:eisenstein-definition}
  E_{k, v}(\tau)
:=
  \frac{1}{\big[ \Gamma_\infty \,:\, \Gamma_\infty(v) \big]}
  \sum_{\gamma \in \Gamma_\infty(v) \backslash \SL{2}(\ZZ)}
  v \big|_{k, \rho}\, \gamma
\end{gather}
is well-defined, converges, and defines a modular form of weight~$k$ and type~$\rho$.  We define the Eisenstein subspace of $\rmM_k(\rho)$ as their span.
\begin{gather}
\label{eq:eisenstein-space}
  \rmE_k(\rho)
=
  \lspan \big\{ E_{k, v} \,:\, v \in V(\rho) \big\}
\text{.}
\end{gather}

In the case of weight~$2$, the Hecke trick leads us to the definition
\begin{gather}
  E_{2, v}(\tau)
:=
  \frac{1}{\big[ \Gamma_\infty \,:\, \Gamma_\infty(v) \big]}\,
  \lim_{s \ra 0} \sum_{\gamma \in \Gamma_\infty(v) \backslash \SL{2}(\ZZ)}
  y^s v \big|_{2, \rho}\, \gamma
\,\tx{,}
\end{gather}
where $y = \Im\,\tau > 0$.
\begin{lemma}
\label{la:weight-2-eisenstein-series}
Assume that $\rho$ has finite index kernel.  The Eisenstein series~$E_{2,v}$ is holomorphic if $\rho$ does not contain the trivial representation.  We have $\rmM_2(\bbone) = \{ 0 \}$ for $\bbone$ the trivial representation of $\SL{2}(\ZZ)$.
\end{lemma}
\begin{proof}
We use the $\xi_2$ operator, first defined in~\cite{bruinier-funke-2004}, to prove the first part. It is defined by
\begin{gather*}
  \xi_2 (E_{2, v})
:=
  2 i y^2\, \ov{\partial_{\ov \tau} E_{2, v}}
\in
  \rmM_0(\ov{\rho})
\,\text{.}
\end{gather*}
By the intertwining property of $\xi_2$, we find that $\xi_2(E_{2,v})$ is a vector valued modular form of weight~$0$ and type $\ov\rho$. Assume that $\rho$ does not contain $\bbone$. Then neither does $\ov\rho$. Lemma~\ref{la:weight-0-modular-forms} therefore implies that $\xi_2(E_{2, v})$ vanishes. Since the kernel of $\xi_2$ applied to $C^\infty(\HS)$ consists of holomorphic functions, this proves that $E_{2,v}$ is holomorphic. The second part is a classical fact.
\end{proof}
\begin{lemma}
\label{la:weight-0-modular-forms}
Suppose that $\rho$ is a representation with finite index kernel. If $\rho$ does not contain the trivial representation, then $\rmM_0(\rho) = \{ 0 \}$.
\end{lemma}
\begin{proof}
Without loss of generality, we may assume that $\rho$ is irreducible and not equal to the trivial representation. Suppose that there is $0 \ne f \in \rmM_0(\rho)$. It is a classical fact that weight~$0$ modular forms for finite index subgroups of $\SL{2}(\ZZ)$ are constant. In particular, we can view $f$ as a nonzero vector in $V(\rho)$, which is invariant under the action of $\SL{2}(\ZZ)$, contradicting irreducibility of~$\rho$.
\end{proof}

Lemma~\ref{la:weight-2-eisenstein-series} inspires the definition
\begin{gather*}
  \rmE_2(\rho)
=
  \big\{ E_{2,v} \,:\, v \in V(\rho') \big\}
\,\text{,}\quad\text{where}\quad
  \rho = \rho(\bbone) \oplus \rho' 
\text{.}
\end{gather*}

The matrix $T = \left(\begin{smallmatrix} 1 & 1 \\ 0 & 1 \end{smallmatrix}\right)$ acts on the upper half space by translations $\tau \mto \tau + 1$. In the context of vector valued modular forms, $T$-eigenspaces of a representation $\rho$ encode information about possible exponents in the Fourier expansion. We describe vector valued Eisenstein series using this information. Write
\begin{gather}
  V(\rho)(1)_T
:=
  \big\{ v \in V(\rho) \,:\, \rho(T) v = v \big\}
\end{gather}
for the isotrivial component of the restriction of $\rho$ to the subgroup of $\SL{2}(\ZZ)$ generated by~$T$.
\begin{proposition}
\label{prop:eisenstein-series-and-T-isotrivial-components}
We have, for $k \ge 2$, 
\begin{gather*}
  \rmE_k(\rho)
\cong
  V(\rho)(1)_T
\quad
\text{and}
\quad
  \rmE_2(\rho)
\cong
  V(\rho')(1)_T
\,\text{,}
\end{gather*}
where $\rho = \rho(\bbone) \oplus \rho'$ as above.
\end{proposition}
\begin{proof}
We have $\Gamma_\infty(v) = \Gamma_\infty(N) := \Gamma_\infty \cap \Gamma(N)$ for some $N$.  Summation over $\Gamma_\infty(N) \backslash \Gamma_\infty$ in the definition of Eisenstein series corresponds to the projection onto $V(\rho)(1)_T \subseteq V(\rho)$, where $V(\rho)$ is decomposed with respect to eigenvalues of~$T$.
\end{proof}

\subsection{Hecke operators acting on Eisenstein series}

Our proof of the Main Theorem requires the following statement.
\begin{proposition}
\label{prop:hecke-operator-on-eisenstein-series}
Let $k \ge 2$ be an even integer and $\rho$ a representation of $\SL{2}(\ZZ)$ with finite index kernel.  For every positive integer~$M$, we have
\begin{gather*}
  T_M \rmE_k(\rho) 
\subseteq
  \rmE_k(T_M\, \rho)
\text{.}
\end{gather*}
\end{proposition}
\begin{proof}
To reduce technicalities, we focus on the case $k > 2$, so that Eisenstein series can be defined without the Hecke trick.  It suffices to consider Eisenstein series attached to $v \in V(\rho)(1)_T$.  Fix such $v$, set $T_M\,v = \sum_{m \in \Delta_M} d(m)^{-k}\, v \otimes \frake_m$, and consider the Eisenstein series
\begin{gather*}
  E_{k, T_M\, v}
=
  \frac{1}{M}
  \sum_{m \in \Delta_M}
  d(m)^{-k}
  \sum_{\gamma \in \Gamma_\infty(M) \backslash \SL{2}(\ZZ)}
  v \otimes \frake_m \big|_{k, T_M\,\rho}\,\gamma
\text{.}
\end{gather*}
Its $m'$\thdash\ component is equal to
\begin{gather*}
  M^{-\frac{k}{2} - 1}
  \sum_{\substack{m \in \Delta_M\\ \gamma \in \Gamma_\infty(M) \backslash \SL{2}(\ZZ)\\ m\gamma = \gamma' m'}}
  \big( v \big|_k\,m \otimes \frake_m \big) \big|_{k, T_M\,\rho}\,\gamma
=
  M^{-\frac{k}{2} - 1}
  \sum_{\substack{m \in \Delta_M\\ \gamma \in \Gamma_\infty(M) \backslash \SL{2}(\ZZ)\\ m\gamma = \gamma' m'}}
  \big( \rho^{-1}(\gamma') v \big|_k\,\gamma' m' \big) \otimes \frake_{m'} 
\text{.}
\end{gather*}
For every $\gamma' \in \Gamma_\infty(M) \backslash \SL{2}(\ZZ)$, we find unique $m$ and $\gamma$ such that $m \gamma = \gamma' m'$.  Therefore, the above simplifies in the following way:
\begin{gather*}
  M^{-\frac{k}{2} - 1}
  \sum_{\gamma' \in \Gamma_\infty(M) \backslash \SL{2}(\ZZ)}
  \big( v \big|_{k, \rho}\,\gamma' \big) \big|_k \,m' \otimes \frake_{m'}
\,\text{,}
\end{gather*}
which equals the $m'$\nbd component of $M^{-\frac{k}{2}}\, T_M\,E_{k, v}$. Since $m' \in \Delta_M$ was arbitrary, we have shown that
\begin{gather*}
  T_M\,E_{k, v}
=
  M^{\frac{k}{2}} 
  E_{k, T_M\, v}
\,\tx{,}
\end{gather*}
which implies the proposition.
\end{proof}

\subsection{Two particular Eisenstein series}
\label{ssec:level-1-eisenstein-series}

In the proof the main theorem, we need two particular Eisenstein series which we introduce now. For $w \in \CC$ and two Dirichlet characters $\delta$ and $\epsilon$, we set
\begin{gather}
  \sigma_{w, \delta, \epsilon}(n)
=
  \sum_{0 < d \isdiv n} \delta(d) \epsilon(n\slash d)\,d^w
\,\tx{.}
\end{gather}
The trivial Dirichlet character mod~$N$ will be denoted by~$\bbone_N$. If $N = 1$, then we instead write $\bbone$.

To state the next two lemmas, recall from Section~\ref{ssec:hecke-operators-known-constructions} our notations $\rho_N = \Ind_{\Gamma_0(N)}$ and $\rho_\chi = \Ind_{\Gamma_0(N)}\,\chi$ for a Dirichlet character~$\chi$ mod~$N$.
\begin{lemma}
\label{la:fourier-expansion-of-eisenstein-series}
Given an even integer $k \ge 4$ and an even Dirichlet characters $\chi$ mod~$N > 1$, there is an Eisenstein series in $\rmE_k(\rho_\chi)$ whose $\Gamma_0(N)$\thdash\ component (i.e.\ the component that corresponds to the trivial coset in $\Gamma_0(N) \slash \SL{2}(\ZZ)$) has Fourier expansion
\begin{gather*}
  \sum_{n=1}^\infty
  \sigma_{k-1,\chi,\bbone}(n)\, q^n
\,\tx{.}
\end{gather*}
\end{lemma}
\begin{proof}
Miyake~\cite{miyake-1989} in his Theorem~7.1.3 computes the Fourier expansions of an Eisenstein series (which he denotes by~$E_k(z; \chi, \bbone)$), which is a modular form for the character~$\chi$ by Lemma~7.1.1. Applying the induction map~\eqref{eq:def:induction-on-modular-forms} to Miyake's construction yields exactly the searched for Eisenstein series.
\end{proof}

\begin{lemma}
\label{la:eisenstein-series-nonzero-at-one-cusp}
Fix an even integer $k \ge 4$ and an even Dirichlet character $\chi$ mod~$N$. There is an adjoint pair of inclusion and projection
\begin{alignat}{3}
\label{la:eisenstein-series-nonzero-at-one-cusp:iota-pi}
  \iota_{\diag, \chi} &:\;&
  \rho_N
&\lra
  \rho_\chi \otimes \rho_{\ov\chi}
\,\tx{,}\quad
&
  \frake_{\gamma}
&\lmto
  \frake_{\gamma} \otimes \frake_{\gamma}
\\
  \pi_{\diag, \chi} &:\;&
  \rho_\chi \otimes \rho_{\ov\chi}
&\lra
  \rho_N
\,\tx{,}\quad
&
  \frake_{\gamma} \otimes \frake_{\gamma'}
&\lmto
  \begin{cases}
    \frake_{\gamma}
  \,\tx{,}
  & \tx{if $\gamma = \gamma'$;}
  \\
    0
  \tx{,}
  & \tx{otherwise.}
  \end{cases}
\end{alignat}
Let $v = \frake_{\Gamma_0(N)} \otimes \frake_{\Gamma_0(N)}  \in V\big( \rho_{\chi} \otimes \rho_{\ov\chi} \big)$. Then the Eisenstein series~$E_{k,\chi,\infty} := E_{k, v}$ satisfies
\begin{gather*}
  E_{k,\chi,\infty}
\in
  \iota_{\diag, \chi}\big( \sum_{M \isdiv N} T_{M}\,\rmE_k(\bbone) \big)
\tx{.}
\end{gather*}
\end{lemma}
\begin{proof}
It is a straight forward computation to verify that $\iota_{\diag, \chi}$ is an inclusion, and that it is adjoint to $\pi_{\diag, \chi}$. To establish the remaining part of the lemma, is suffices to show that the Eisenstein series $E_{k, v'}$ for $v' = \frake_{\Gamma_0(N)} \in V(\rho_N)$ lies in $\sum_{M \isdiv N} T_M\, \rmE_k(\bbone)$. Observe that $E_{k,v'}$ corresponds, under the map~(\ref{eq:def:induction-on-modular-forms}), to an Eisenstein series that vanishes at every cusp but $\infty$. Such a series is an oldform coming from level~$1$, by classical theory. Apply Proposition~\ref{prop:classical-oldforms}, which says that oldforms can be constructed using vector valued Hecke operators, to finish the proof.
\end{proof}

\subsection{A pairing of modular forms}

Let $\rho$ and $\rho_\rmE$ be two representations of $\SL{2}(\ZZ)$ with finite index kernel.  For even~$l$ satisfying $2 \le l \le k - 2$, we define
\begin{gather*}
  F :\;
  \rmE_{k - l}(\rho_\rmE) \otimes V(\ov{\rho} \otimes \rho_\rmE)^\vee
\lra
  \Hom\big(\rmS_{k}(\rho), \CC \big)
\text{,}\quad
  E \otimes v
\lmto
  \big(f \mapsto \big\langle f,\, E \otimes E_{l, v} \big\rangle_\iota \big)
\,\text{,}
\end{gather*}
where the scalar product is taken with respect to the canonical inclusion~$\iota$ of the trivial representation into
\begin{gather*}
  \rho \otimes \ov{\rho_\rmE} \otimes \big( \rho \otimes \ov{\rho_\rmE} \big)^\vee
=
  \rho
  \otimes
  \ov{\rho_\rmE \otimes \big( \ov\rho \otimes \rho_\rmE \big)^\vee}
\tx{.}
\end{gather*}

\begin{proposition}
\label{prop:evaluation-of-pairing}
For $k \ge 4$, $2 \le l \le k - 2$, and two $\SL{2}(\ZZ)$\nbd representations $\rho$ and $\rho_E$, let $E \in \rmE_{k-l}(\rho_E)$, $f \in \rmS_k(\rho)$, and $v \in V(\ov{\rho} \otimes \rho_E)^\vee(1)_T$. Under these assumptions we have
\begin{gather}
  F(E \otimes v)(f)
=
 \frac{\Gamma(k-1)}{(4 \pi)^{k - 1}}\,
  \sum_{0 \le n \in \QQ} n^{1 - k}\;
  \ov{v}\big( c(f;\, n) \otimes \ov{c(E;\, n)} \big)
\tx{,}
\end{gather}
where $\ov v$ is the complex conjugate of~$v$.
\end{proposition}
\begin{proof}
For simplicity, we assume that $l, k - l > 2$, so that we need not apply the Hecke trick.  The reader readily verifies that all arguments remain valid after introducing auxiliary variables $s_1$ and $s_2$ that tend to~$0$ and factors $y^{s_1}$, $y^{s_2}$.

Set $\sigma = (\ov{\rho} \otimes \rho_\rmE)^\vee$. We unfold the Petersson scalar product to obtain an explicit formula for $\langle f,\, E \otimes E_{l, v} \rangle_\iota$:
\begin{align*}
&
  \int_{\SL{2}(\ZZ) \backslash \HS}
  \Big\langle
  \big( f \otimes \ov{( E \otimes E_{l, v} )} \big)(\tau),\, \iota(1)
  \Big\rangle
  \;\frac{d\!x\,d\!y}{y^{2-k}}
=
  \int_{\SL{2}(\ZZ) \backslash \HS}
  \Big\langle
  \Big( f \otimes
  \ov{\big( E \otimes \sum_\gamma v \big|_{l, \sigma}\, \gamma \big)}
  \Big)(\tau),\, \iota(1)
  \Big\rangle
  \;\frac{d\!x\,d\!y}{y^{2-k}}
\\
={}&
  \int_{\SL{2}(\ZZ) \backslash \HS}
  \sum_\gamma 
  \big( \ov{v \big|_{l, \sigma}\, \gamma} \big)
  \big( f \otimes \ov E \big)(\tau)
  \;\frac{d\!x\,d\!y}{y^{2-k}}
=
  \int_{\SL{2}(\ZZ) \backslash \HS}
  \sum_\gamma 
  \big( \ov{v \big|_{l, \sigma}\, \gamma} \big)
  \big( f \otimes \ov E \big|_{k-l, \rho \otimes \ov{\rho_E}}\, \gamma \big)(\tau)
  \;\frac{d\!x\,d\!y}{y^{2-k}}
\\
={}&
  \int_{\SL{2}(\ZZ) \backslash \HS}
  \sum_\gamma
  \Big( \ov{v} \big( f \otimes \ov E \big)(\tau) \Big) \Big|_k\, \gamma
  \;\frac{d\!x\,d\!y}{y^{2-k}}
=
  \int_{\Gamma_\infty \backslash \HS}
  \ov{v} \Big( \big( f \otimes \ov E \big)(\tau) \Big)
  \;\frac{d\!x\,d\!y}{y^{2-k}}
\,\tx{.}
\end{align*}
Since $f$ is a cusp form, all the displayed integrals converge absolutely. In the first equality we have inserted the definition of~$E_{l,v}$. The second equality holds by virtue of the definition of~$\iota$. More precisely, we have $\iota(1) = \sum_w w \otimes w^\vee$ where $w$ runs through an orthonormal basis of $V(\rho \otimes \ov{\rho_\rmE})$ and $w^\vee$ denotes a corresponding dual basis element. For the third equality, we have employed modular invariance of $f \otimes \ov E$. The fourth equality follows from $(\sigma(\gamma) v)(\sigma^\vee(\gamma) w) = v(w)$, which is true for all $w \in V(\sigma)^\vee$. The final equality is based on $\SL{2}(\ZZ)$\nbd invariance of the measure $y^{-2} d\!x d\!y$.

As our next step, we carry out the integral with respect to $x$. Expanding $f \otimes E$ into a Fourier series, we obtain
\begin{gather*}
  \ov v \big( (f \otimes \ov E)(\tau) \big)
=
  \sum_{n,n'}
  \ov v \big( c(f;\,n) \otimes \ov{c(E;\,n')} \big)\,
  \exp\big( 2 \pi i (n \tau - n' \ov\tau) \big)
\,\tx{.}
\end{gather*}
The integral with respect to~$x$ picks up terms with $n = n'$. We may interchange integration and summation, since the resulting right hand side converges absolutely. We therefore obtain
\begin{gather*}
  \langle f,\, E \otimes E_{l, v} \rangle
=
  \int_0^\infty
  \sum_n
  \ov v \big( c(f;\,n) \otimes \ov{c(E;\,n)} \big)\,
  \exp(- 4 \pi\, n y)
  \;\frac{d\!y}{y^{2 - k}}
=
  \frac{\Gamma(k-1)}{(4 \pi)^{k - 1}}\!
  \sum_{0 \le n \in \QQ}\!\! n^{1 - k}\,
  v\big( c(f;\, n) \otimes \ov{c(E;\, n)} \big)
\,\tx{.}
\end{gather*}
\end{proof}

\section{Products of Eisenstein series}

\subsection{Rankin convolutions}
\label{ssec:rankin-convolutions}

The next proposition generalizes Rankin's~\cite{rankin-1952} statement about $L$\nbd series of modular forms for~$\SL{2}(\ZZ)$.
\begin{proposition}
\label{prop:product-decomposition-of-convolution}
Let $f \in \rmM_k(\chi)$ be a newform with Fourier coefficients $c(n)$.  Then for Dirichlet characters $\delta$ and $\epsilon$, and two complex variables $s, w$ with $\Re s > \Re w + \frac{k+3}{2}$, we have
\begin{gather*}
  \sum_{n = 1}^\infty \sigma_{w, \delta, \epsilon}(n) c(n)\, n^{-s}
=
  \frac{L(f \times \epsilon, s) L(f \times \delta, s - w)}
       {L(\chi \delta \epsilon, 2s - w + 1 - k)}
\,\text{.}
\end{gather*}
\end{proposition}
\begin{proof}
The proof is completely analogous to the one in~\cite{rankin-1952}. By assumptions, we have for $m, n \in \ZZ$ that
\begin{gather*}
  c(m) c(n)
=
  \sum_{d \isdiv (m,n)} \chi(d) d^{k-1}\, c\big(\frac{mn}{d^2} \big)
\text{.}
\end{gather*}
Expanding $L(f \times \epsilon, s) L(f \times \delta, s - w)$, then applying that relation, and finally simplifying, we obtain
\begin{align*}
  \sum_{\substack{m,n\\d \isdiv (m, n)}}
  \chi(d) \epsilon(m) \delta(n)\,
  c\big( \frac{mn}{d^2} \big)\, d^{k-1} (mn)^{-s} n^w
=  \sum_d \chi\delta\epsilon(d) d^{k - 1 + w - 2s}\,
  \sum_{m,n}  \epsilon(m) \delta(n) c(mn) (mn)^{-s} n^w
\,\text{.}
\end{align*}
This equals the desired factorization
\begin{gather*}
  L(\chi\delta\epsilon, 2s - w + 1 - k)\,
  \sum_{m;\, n \isdiv m}  \epsilon(\tfrac{m}{n}) \delta(n) c(m) m^{-s} n^w
=
  L(\chi\delta\epsilon, 2s - w + 1 - k)\,
  \sum_{m}  \sigma_{w, \delta, \epsilon}(m) c(m) m^{-s}
\,\tx{.}
\end{gather*}
\end{proof}

\subsection{Proof of the Main Theorem}

Throughout the subsection, we fix positive, even integers~$k$ and~$l$ such that $l, k-l \ge 4$. Set $\rho_{T_N} = T_N\,\bbone$ for any $N$. Note that these are the same representations $\rho_{T_N}$ that we referred to in the introduction. For a finite dimensional complex representation $\rho$ of $\SL{2}(\ZZ)$ we define the statement ${\rm Span}(\rho)$ as
\begin{gather}
\label{eq:proof-of-main-theorem:Span-statement}
  {\rm Span}(\rho) \,:\;
  \rmM_k(\rho)
=
  \rmE_k(\rho)
  \;+
  \lspan_{\substack{ 0 < N, N' \in \ZZ \\ \phi :\, \rho_{T_N} \otimes \rho_{T_{N'}} \ra \rho}}\!\!
  \phi\Big(\;  T_N\,\rmE_{l}(\bbone) \otimes T_{N'}\,\rmE_{k - l}(\bbone) \; \Big)
\tx{.}
\end{gather}
The Main Theorem says that ${\rm Span}(\rho)$ is true for all $\rho$ whose kernel is a congruence subgroup. It is further convenient to abbreviate
\begin{gather}
\label{eq:proof-of-main-theorem:right-hand-side-EE}
  \rmE\rmE_k(\rho)
:=
  \rmE_k(\rho)
  \,+\,
  \lspan_{\substack{ 0 < N, N' \in \ZZ \\ \phi :\, \rho_{T_N} \otimes \rho_{T_{N'}} \ra \rho}}\!\!
  \phi\Big(\; T_N\,\rmE_{l}(\bbone) \otimes T_{N'}\,\rmE_{k - l}(\bbone) \;\Big)
\tx{.}
\end{gather}
Since the weight is typically clear from the context, we will often suppress the subscript, writing $\rmE\rmE(\rho)$.

\begin{lemma}
\label{la:EE-stable-under-morphisms-and-hecke}
For any homomorphism of representations $\psi:\, \rho \ra \rho'$ and for any positive integer~$M$ we have
\begin{gather*}
  \psi\big( \rmE\rmE(\rho) \big)
\subseteq
  \rmE\rmE(\rho')
\quad\tx{and}\quad
  T_M\big( \rmE\rmE(\rho) \big)
\subseteq
  \rmE\rmE\big( T_M\, \rho \big)
\tx{.}
\end{gather*}
\end{lemma}
\begin{proof}
To prove the first equality observe that
\begin{multline*}
  \psi\big( \rmE\rmE(\rho) \big)
=
  \psi\Bigg(
  \rmE_{k}(\rho)
  \,+\,
  \lspan_{\substack{ 0 < N, N' \in \ZZ \\
                     \phi :\, \rho_{T_N} \otimes \rho_{T_{N'}} \ra \rho}}\!\!
  \phi\Big(\; T_N\,\rmE_{l}(\bbone) \otimes T_{N'}\,\rmE_{k - l}(\bbone) \;\Big)
  \Bigg)
\\
=
  \psi\big( \rmE_{k}(\rho) \big)
  \,+\,
  \lspan_{\substack{ 0 < N, N' \in \ZZ \\
                     \phi :\, \rho_{T_N} \otimes \rho_{T_{N'}} \ra \rho}}\!\!
  \psi \circ \phi\Big(\; T_N\,\rmE_{l}(\bbone) \otimes T_{N'}\,\rmE_{k - l}(\bbone) \;\Big)
\tx{.}
\end{multline*}
It is clear from the definition of Eisenstein series~\eqref{eq:eisenstein-definition} that $\psi\big( \rmE_{k}(\rho) \big) \subseteq \rmE_k(\rho')$. Further, composition with $\psi$ yields a map
\begin{gather*}
  \psi_\ast :\,
  \Hom\big(T_N\,\bbone \otimes T_{N'}\,\bbone,\, \rho \big)
\lra
  \Hom\big(T_N\,\bbone \otimes T_{N'}\,\bbone,\, \rho' \big)
\tx{.}
\end{gather*}
We thus obtain the desired inclusion into $\rmE\rmE(\rho')$.

We now address the case of Hecke operators. The inclusion of spaces of Eisenstein series~$T_M\, \rmE_k(\rho) \subseteq \rmE_k(T_M\,\rho)$ is stated in Proposition~\ref{prop:hecke-operator-on-eisenstein-series}. We consider the space
\begin{gather*}
  T_M\Big(
  \phi \big( T_N \rmE_l(\bbone) \otimes T_{N'} \rmE_{k-l}(\bbone) \big)
  \Big)
\tx{.}
\end{gather*}
Recall that by Proposition~\ref{prop:hecke-operator-comonoidal-on-reps}, for each $\phi \in \Hom( T_N\,\bbone \otimes T_{N'}\,\bbone,\, \rho)$ there is a homomorphism
\begin{gather*}
  T_M\,\phi :\,
  T_M \big( T_N\,\bbone \otimes T_{N'}\,\bbone \big)
\lra
  T_M\,\rho
\end{gather*}
that satisfies $T_M \circ \phi = (T_M\,\phi) \circ T_M$. Therefore, we have
\begin{multline*}
  T_M\Big(
  \phi \big( T_N \rmE_l(\bbone) \otimes T_{N'} \rmE_{k-l}(\bbone) \big)
  \Big)
=
  (T_M\,\phi)\Big(
  T_M \big( T_N \rmE_l(\bbone) \otimes T_{N'} \rmE_{k-l}(\bbone) \big)
  \Big)
\\
\subseteq
  (T_M\,\phi)\Big( \pi \Big(
  T_{MN} \rmE_l(\bbone) \otimes T_{MN'} \rmE_{k-l}(\bbone) \big)
  \Big) \Big)
\tx{,}
\end{multline*}
where the last inclusion is a direct consequence of Theorem~\ref{thm:inclusion-of-hecke-operator-tensor-product-on-modular-forms} and
\begin{gather*}
  \pi :\,
  T_{MN}\, \bbone \otimes T_{M N'}\, \bbone
\lthra
  T_M \big( T_N\,\bbone \otimes T_{N'}\,\bbone \big)
\end{gather*}
is the projection displayed in~\eqref{eq:prop:hecke-operator-comonoidal-on-reps:comonoidal-coherence}. Inserting this into the definitions of $\rmE\rmE(\rho)$ and $\rmE\rmE(T_M\,\rho)$, we prove the lemma.
\end{proof}

\subsubsection{Two lemmas about~${\rm Span}$}

We now establish two lemmas that will allows us to focus on the case $\rho = \Ind_{\Gamma_0(N)}\,\bbone$ in the actual proof of the main theorem. 
\begin{lemma}
\label{la:proof-of-main-theorem:decomposing-representations}
Given representations $\rho$ and $\rho'$, we have ${\rm Span}(\rho) \wedge {\rm Span}(\rho') \Longleftrightarrow {\rm Span}(\rho \oplus \rho')$.
\end{lemma}
\begin{proof}
Proposition~\ref{prop:decomposition-of-modular-forms-spaces} says that $\rmM_k(\rho \oplus \rho') = \rmM_k(\rho) \oplus \rmM_k(\rho')$. Proposition~\ref{prop:eisenstein-series-and-T-isotrivial-components} implies the analogue for Eisenstein series: We have $\rmE_k(\rho \oplus \rho') = \rmE_k(\rho) \oplus \rmE_k(\rho')$.

For simplicity define
\begin{gather*}
  \rmH(N,N')
=
  \Hom\big( \rho_{T_N} \otimes \rho_{T_{N'}},\, \rho \big)
\quad\tx{and}\quad
  \rmH'(N,N')
=
  \Hom\big( \rho_{T_N} \otimes \rho_{T_{N'}},\, \rho' \big)
\tx{.}
\end{gather*}
Now, assume that ${\rm Span}(\rho)$ and ${\rm Span}(\rho')$ both are true. Then
\begin{multline*}
  \rmM_k(\rho \oplus \rho')
=
  \rmM_k(\rho) \oplus \rmM_k(\rho')
\\
=
  \Bigg(
  \rmE_k(\rho)
  +
  \lspan_{\phi \in \rmH(N,N')}
  \phi \Big( T_N\, \rmE_l(\bbone) \otimes T_{N'} \rmE_{k-l}(\bbone) \Big)
  \Bigg)
  \,\oplus\,
  \Bigg(
  \rmE_k(\rho')
  +
  \lspan_{\phi' \in \rmH'(N,N')}
  \phi' \Big( T_N\, \rmE_l(\bbone) \otimes T_{N'} \rmE_{k-l}(\bbone) \Big)
  \Bigg)
\\
= 
  \rmE_k(\rho \oplus \rho')
  +
  \lspan_{\phi \oplus \phi' \in \rmH(N,N') \oplus \rmH'(N,N')}
  (\phi \oplus \phi') \Big( T_N\, \rmE_l(\bbone) \otimes T_{N'} \rmE_{k-l}(\bbone) \Big)
\tx{.}
\end{multline*}
Since $\Hom\big( \rho_N \otimes \rho_{N'},\, \rho \oplus \rho' \big) = \rmH(N,N') \oplus \rmH'(N,N')$, this shows that ${\rm Span}(\rho \oplus \rho')$ is true.

Conversely, if ${\rm Span}(\rho \oplus \rho')$ is  true, then apply the canonical projections $\pi : \rho \oplus \rho' \ra \rho$ and $\pi' : \rho \oplus \rho' \ra \rho'$ to both sides of the equality
\begin{gather*}
  \rmM_k(\rho \oplus \rho')
=
  \rmE_k(\rho \oplus \rho')
  +
  \lspan_{\phi \oplus \phi' \in \rmH(N,N') \oplus \rmH'(N,N')}
  (\phi \oplus \phi') \Big( T_N\, \rmE_l(\bbone) \otimes T_{N'} \rmE_{k-l}(\bbone) \Big)
\end{gather*}
to find that also ${\rm Span}(\rho)$ and ${\rm Span}(\rho')$ are true.
\end{proof}

\begin{lemma}
\label{la:proof-of-main-theorem:employ-intertwining}
Fixing a congruence subgroup~$\Gamma$ and a character~$\chi$ of~$\Gamma$, suppose that
\begin{gather*}
  \rmM_k(\Gamma, \chi)
=
  \sum_i \psi_i \big( \rmV\rmM_i \big)
\tx{,}\qquad
  \rmV\rmM_i
\subseteq
  \rmM_k(\Gamma_i, \chi_i) 
\end{gather*}
for a finite number of congruence subgroups~$\Gamma_i$, characters~$\chi_i$ of $\Gamma_i$, and maps
\begin{gather*}
  \psi_i :\,
  \rmM_k(\Gamma_i, \chi_i)
\lra
  \rmM_k(\Gamma, \chi)
\tx{.}
\end{gather*}
If for every $i$, there is a positive integer~$M_i$ and a projection
\begin{gather*}
  \pi_i :\,
  T_{M_i}\big( \Ind_{\Gamma_i}\,\chi_i \big)
\thra
  \Ind_{\Gamma}\,\chi
\quad \tx{such that}\quad
  \Big\langle \Ind\,\psi_i(f),\, v \Big\rangle
=
  \Big\langle \pi\big( T_{M_i}\, \Ind(f) \big),\, v \Big\rangle
\end{gather*}
for all $v \in \Ind_\Gamma\, \chi$, then
\begin{gather*}
  \Big(
  \forall i \,:\,
  \Ind\, \rmV\rmM_i
  \subseteq
  \rmE\rmE\big( \Ind_{\Gamma_i}\,\chi_i \big)
  \Big)
\;\Longrightarrow\;
  {\rm Span}\big( \Ind_{\Gamma}\,\chi \big)
\tx{.}
\end{gather*}
\end{lemma}
\begin{proof}
We apply the map $\Ind$ defined in~\eqref{eq:def:induction-on-modular-forms} to the assumption
\begin{gather*}
  \rmM_k(\Gamma, \chi)
\subseteq
  \sum_i \psi_i \big( \rmV\rmM_i \big)
\,\tx{,}\quad\tx{implying that}\quad
  \rmM_k(\Ind_\Gamma\, \chi)
\subseteq
  \sum_i \Ind \Big( \psi_i \big( \rmV\rmM_i \big) \Big)
\tx{.}
\end{gather*} 
On the right hand side we can employ the intertwining properties of the projections~$\pi_i$ to obtain
\begin{gather*}
  \rmM_k(\Ind_\Gamma\, \chi)
\subseteq
  \sum_i \pi_i \circ T_M\big( \Ind\, \rmV\rmM_i \big)
\tx{.}
\end{gather*}

Let us now assume that $\Ind\, \rmV\rmM_i \subseteq \rmE\rmE\big( \Ind_{\Gamma_i}\,\chi_i \big)$ for all $i$, and let us deduce that this implies, as desired, that ${\rm Span}(\Ind_\Gamma\,\chi)$ is true. Continuing our previous chain of inclusions, we find that 
\begin{gather*}
  \rmM_k(\Ind_\Gamma\, \chi)
\subseteq
  \sum_i \pi_i \circ T_M\Big( \rmE\rmE_k\big( \Ind_{\Gamma_i}\, \chi_i \big) \Big)
\subseteq
  \sum_i \pi_i \Big( \rmE\rmE_k\big( T_M\, \Ind_{\Gamma_i}\, \chi_i \big) \Big)
\subseteq
  \sum_i \rmE\rmE_k\Big( \pi_i\, T_M\big( \Ind_{\Gamma_i}\, \chi_i \big) \Big)
\tx{.}
\end{gather*}
The last and next to last inclusions follow from Lemma~\ref{la:EE-stable-under-morphisms-and-hecke}, which says that the spaces $\rmE\rmE$ are stable under the application of representation homorphisms and Hecke operators. It is part of the assumptions that $\pi_i T_{M_i}\big( \Ind_{\Gamma_i}\,\chi_i \big)$ equals $\Ind_\Gamma\,\chi$, so that we obtain that
\begin{gather*}
  \rmM_k(\Ind_\Gamma\, \chi)
\subseteq
  \sum_i \rmE\rmE_k\big( \Ind_\Gamma\, \chi \big)
=
  \rmE\rmE_k\big( \Ind_\Gamma\, \chi \big)
\tx{.}
\end{gather*}
\end{proof}

\subsubsection{Proof of the Main Theorem: Reduction to the case of $\rho = \rho_M$}

We now start to actually prove the Main Theorem. Using the tools developed in the previous two lemmas, we show that it suffices to establish ${\rm Span}(\rho_M)$ for all $M$ in order to settle all other cases.

Recall the assumptions: $\rho$ is a representation of~$\SL{2}(\ZZ)$ whose kernel is a congruence subgroup. Further, $k, l$ are even integers satisfying $l, k -l \ge 4$.

We carry out the reduction in five steps.
\begin{enumeratearabic}
\item
Since $\rho$ is finite dimensional and has finite index kernel, it is unitary. Therefore $\rho = \bigoplus_i \rho_i$ for a finite number of irreducible $\rho_i$. By Lemma~\ref{la:proof-of-main-theorem:decomposing-representations}, it thus suffices to treat irreducible $\rho$.

\item
By the same argument, we can prove the theorem for $\rho = \Ind_{\Gamma(M)} \bbone$ for all positive $M$ instead of irreducible~$\rho$. Indeed, suppose that, given an irreducible~$\rho$, we have $\Gamma(M) \subseteq \ker\,\rho$ for some~$M$, then $\rho \subseteq \Ind_{\Gamma(M)} \bbone$. Lemma~\ref{la:proof-of-main-theorem:decomposing-representations} shows that ${\rm Span}\big( \Ind_{\Gamma(M)} \bbone \big)$ implies ${\rm Span}(\rho)$.

\item
Recall the rescaling operator ${\rm sc}_{1\slash M}$ from~\eqref{eq:def:rescaling-of-arguments}. We have
\begin{gather*}
  \rmM_k(\Gamma(M))
\subseteq
  {\rm sc}_{1 \slash M} \big( \rmM_k(\Gamma_1(M^2,M)) \big)
\tx{.}
\end{gather*}
Proposition~\ref{prop:classical-principal-congruence-subgroups} asserts existence of an inclusion
\begin{gather*}
  \iota_{\Gamma(M)} :
  \rho_{\Gamma(M)}
\lhra
  T_N\, \rho_{\Gamma_1(M^2, M)}
\end{gather*}
that intertwines ${\rm sc}_{1 \slash M}$ and the vector valued Hecke operator $T_M$. Therefore Lemma~\ref{la:proof-of-main-theorem:employ-intertwining} allows us to restrict ourselves to the case $\rho = \rho_{\Gamma_1(M)} = \Ind_{\Gamma_1(M)}\,\bbone$ for arbitrary~$M$ (we have replaced $M^2$ by $M$ in this very last sentence).

\item
We have $\Ind_{\Gamma_1(M)}\,\bbone = \bigoplus_\chi\, \Ind_{\Gamma_0(M)}\,\chi$, where $\chi$ runs through all Dirichlet characters mod~$M$. Lemma~\ref{la:proof-of-main-theorem:decomposing-representations} thus allows us to focus on the case $\rho = \rho_\chi = \Ind_{\Gamma_0(M)}\, \chi$ for arbitrary $M$ and arbitrary~$\chi$.
\end{enumeratearabic}

The fifth reduction step requires an inductive argument. Observe that $k$ is even by the assumptions, and therefore $\rmM_k(\chi) \cong \rmM_k\big( \Ind_{\Gamma_0(M)}\,\chi \big)$ is trivial, if $\chi$ is odd. In other words, we can and will assume that $\chi$ is even. It therefore has a square root: $\chi = \chi^{\prime\, 2}$. Classical theory of modular forms asserts that
\begin{gather}
\label{eq:main-theorem:chi-twisting}
  \rmM_k(\chi)
\subseteq
  {\rm twist}_{\chi'}\big( \rmM_k(N^3) \big)
  +
  \rmM^{\rm old}_k(\chi)
\,\tx{,}\quad
  \rmM^{\rm old}_k(\chi)
=
  \sum_{\substack{M' \isdiv M \\ M' \ne M}}
  \sum_{\chi_{M'}}
  {\rm sc}_{M \slash M'} \big( \rmM_k(\chi_{M'}) \big)
\end{gather}
where $\chi_{M'}$ runs through Dirichlet characters mod~$M'$. Indeed, given $f \in \rmM_k(\chi)$, we consider its twist $f_{\chi'} \in \rmM_k(N \cdot N^2)$. The difference $f - \mathrm{twist}_{\chi'}( f_{\chi'} )$ has Fourier expansion supported on $n$ with $\gcd(n,N) \ne 1$. In other words, it is an old form.

We reduce ourselves to the proving that ${\rm Span}(\rho_M)$ holds for all $M$ by employing Lemma~\ref{la:proof-of-main-theorem:employ-intertwining} in conjunction with induction on~$M$. For the time being assume that ${\rm Span}(\rho_M)$ is true for all $M$. We will establish this in Sections~\ref{ssec:proof-of-main-theorem:atkin-lehner}, \ref{ssec:proof-of-main-theorem:twists}, and~\ref{ssec:proof-of-main-theorem:non-vanishing}---without recursing to the statement ${\rm Span}(\rho_\chi)$ for non-trivial~$\chi$.

Fixing~$M$, our induction hypothesis is that ${\rm Span}(\Ind_{\Gamma_0(M')}\, \chi_{M'})$ holds for all $M' \isdiv M$ with $M' \ne M$ and all Dirichlet characters mod~$M'$. We then show that ${\rm Span}(\Ind_{\Gamma_0(M)}\, \chi)$ for any Dirichlet character~$\chi$ mod~$M$. If $M = 1$ then the only Dirichlet character is the trivial one. This establishes the induction basis. For the induction step, assume that $M > 1$. Proposition~\ref{prop:twists-of-modular-forms} guarantees existence of $\iota_{\rm twist}$ intertwining $T_{M}$ and the twisting~${\rm twist}_{\chi'}$. Proposition~\ref{prop:classical-oldforms} asserts existence of $\iota_{\rm old}$ intertwining $T_{M \slash M'}$ and the old form construction ${\rm sc}_{M \slash M'}$. Combine this with the inclusion in~\eqref{eq:main-theorem:chi-twisting} to see that the assumptions of Lemma~\ref{la:proof-of-main-theorem:employ-intertwining} are satisfied. The statement ${\rm Span}(\Ind_{\Gamma_0(M)}\,\chi)$ is therefore implied by ${\rm Span}(\Ind_{\Gamma_0(M)}\,\bbone)$ and ${\rm Span}(\Ind_{\Gamma_0(M')}\,\chi_{M'})$ for all $M' \isdiv M$ with $M' \ne M$. The induction hypothesis implies the latter. The former is part of our standing assumption that $\mathrm{Span}(\rho_M)$ be true for all $M$. This finishes our argument.

\subsubsection{Proof of the Main Theorem: Classical Atkin-Lehner-Li Theory}
\label{ssec:proof-of-main-theorem:atkin-lehner}

Summarizing, we have shown so far that is suffices to establish ${\rm Span}(\rho_M)$ for all positive~$M$. More precisely, we have to show that $\rmM_k(\rho_M) = \rmE\rmE_k(\rho_M)$, where $\rmE\rmE_k(\rho)$ was defined in~\eqref{eq:proof-of-main-theorem:right-hand-side-EE} for any~$\rho$. In this Section, we will nevertheless treat $\rho_\chi$ for arbitrary Dirichlet characters~$\chi$, since no additional complications arise from this level of generality.

We can identify $\rmM_k(\rho_\chi)$ with the space of classical modular forms $\rmM_k(\chi)$ via the inverse to $\Ind$ given in~(\ref{eq:def:induction-on-modular-forms}):
\begin{gather*}
  \Ind^{-1} :\,
  \rmM_k(\rho_\chi) \lra \rmM_k(\chi)
\text{,}\quad
  f \lmto \langle f,\, \frake_{\Gamma_0(M)} \rangle
\text{.}
\end{gather*}
Set $\rmE\rmE'_k(\rho_\chi) = \Ind^{-1}\, \rmE\rmE_k(\rho_\chi) \subseteq \rmM_k(\chi)$. As for~$\rmE\rmE$ we suppress the subscript of~$\rmE\rmE'$ except for few cases.

From classical Atkin-Lehner-Li theory we find that
\begin{gather*}
  \rmM_k(\chi)
=
  \bigoplus_{M' \isdiv M}
  {\rm sc}_{M \slash M'} \, \rmM^{\rm new}_k(\Gamma_0(M'),\, \chi)
\,\tx{,}
\end{gather*}
where $\rmM^{\rm new}_k(\Gamma_0(M'), \chi)$ is the space of level~$M'$ newforms for $\chi$ if $\chi$ extends as a character of the group $\Gamma_0(M)$ to $\Gamma_0(M')$, and $\rmM^{\rm new}_k(\Gamma_0(M'), \chi) = \{0\}$ if it does not. Proposition~\ref{prop:classical-oldforms} and Lemma~\ref{la:proof-of-main-theorem:employ-intertwining} show that it suffices to establish that $\rmM^{\rm new}_k(\chi) \subseteq \rmE\rmE'(\rho_\chi)$ for all~$\chi$, in order to deduce~${\rm Span}(\rho_\chi)$.

Our next goal is to show that $\rmE\rmE'(\rho_\chi)$ is a Hecke module with respect to the classical Hecke algebra. That is, it is closed under the action of classical Hecke operators of level coprime to~$M$ and under the action of Atkin-Lehner involutions. For $M'$ coprime to~$M$, we have
\begin{gather*}
  \Ind \big( \rmE\rmE'(\rho_\chi) \big|_k\, T_{M'} \big)
=
  \pi_{\rm Hecke} \big( T_{M'} \Ind\,\rmE\rmE'(\rho_\chi) \big)
=
  \pi_{\rm Hecke} \big( T_{M'}\, \rmE\rmE(\rho_\chi) \big)
\subseteq
  \rmE\rmE\big( T_{M'}\,\rho_\chi \big)
\tx{.}
\end{gather*}
The first equality is stated in Proposition~\ref{prop:classical-hecke-operator}. The second equality follows from the definition of~$\rmE\rmE'$. The inclusion, employed at the end, is proved in Lemma~\ref{la:EE-stable-under-morphisms-and-hecke}. A similar argument shows that $\rmE\rmE'(\rho_\chi)$ is also closed under the action of all Atkin-Lehner involutions.

As a consequence of this, we can focus on newforms. The remainder of the proof is about deducing a contradiction to $\rmM^{\rm new}_k(\Gamma_0(M)) \nsubseteq \rmE\rmE'_k(\rho_M)$. Suppose that indeed $\rmM^{\rm new}_k(\Gamma_0(M)) \nsubseteq \rmE\rmE'_k(\rho_M)$ for some~$M$. Then the orthogonal complement with respect to the regularized scalar product of $\rmE\rmE'_k(\rho_M)$ in $\rmS^{\rm new}_k(\Gamma_0(M))$ in nonempty. Since both $\rmE\rmE'_k(\rho_M)$ and $\rmS^{\rm new}_k(\Gamma_0(M))$ are classical Hecke modules, that orthogonal complement is too. In particular, we find a newform~$f$ such that $\langle f, g \rangle = 0$ for all $\rmE\rmE'(\rho_M)$. We will show that this is impossible.

\subsubsection{Proof of the Main Theorem: Twists of newforms}
\label{ssec:proof-of-main-theorem:twists}

To treat the case $l = k \slash 2$, we have to consider twists of newforms. We claim that
\begin{gather}
\label{eq:proof-of-main-theorem:twist-is-orthogonal}
  \forall g \in \rmE\rmE'(\rho_M) \,:\,
  \langle f, g \rangle = 0
\quad\Longrightarrow\quad
  \forall \epsilon \tx{ Dirichlet character mod~$N$}\;
  \forall g \in \rmE\rmE'(\rho_{\epsilon^2}) \,:\,
  \langle f_\epsilon, g \rangle = 0
\tx{.}
\end{gather}
To see this, we first pass to the induction:
\begin{gather*}
  \langle f_\epsilon, g \rangle
=
  \Big\langle \Ind(f_\epsilon),\, \Ind(g) \Big\rangle
=
  \Big\langle T_N\,\Ind(f),\, \iota_{\rm twist}\big( \Ind(g) \big) \Big\rangle
=
  \Big\langle \Ind(f),\, \pi_{\rm adj}\big( T_N\, \iota_{\rm twist}\big( \Ind(g) \big) \Big\rangle
\tx{.}
\end{gather*}
The second equality is the part of the statement of Proposition~\ref{prop:twists-of-modular-forms}. The third one follows from the adjunction formula in Proposition~\ref{prop:adjoint-of-hecke-operator}. In order to prove that $\langle f_\epsilon, g \rangle = 0$, it remains to be shown that $\pi_{\rm adj} T_N\, \iota_{\rm twist}\, \Ind\,g \in \rmE\rmE$. This is consequence of Lemma~\ref{la:EE-stable-under-morphisms-and-hecke}, saying that $\rmE\rmE$ is mapped to itself by homomorphisms of representations and by Hecke operators.

\subsubsection{The proof of the Main Theorem: Nonvanishing of $L$-functions}
\label{ssec:proof-of-main-theorem:non-vanishing}

We are left with showing that for every newform~$f$ for~$\Gamma_0(M)$ there is a character~$\epsilon$ mod~$N$ such that $\langle f_\epsilon, g \rangle \ne 0$ for at least one $g \in \rmE\rmE'(\rho_{\epsilon^2})$. By interchanging the role of the first and second factor in the definition of $\rmE\rmE$, we may assume that $l \le k \slash 2$. Fix a negative fundamental discriminant~$D$, and consider the Kronecker character~$\epsilon_D$ and its square $\bbone_{|D|} = \epsilon_D^2$, which is a trivial, non-primitive Dirichlet character. In particular, the classical Eisenstein series
\begin{gather*}
  E_{l, \bbone_{|D|}}(\tau)
=
  \sum_{n=1}^\infty
  \sigma_{k-1,\bbone_{|D|},\bbone}(n)\, q^n
\end{gather*}
considered in Lemma~\ref{la:fourier-expansion-of-eisenstein-series} is an oldform that comes from level one. Further, recall the Eisenstein series
\begin{gather*}
  E_{k-l,\bbone_{|D|},\infty}
=
  \sum_{\gamma \in \Gamma_\infty \backslash \SL{2}(\ZZ)}
  \frake_{\Gamma_0(|D|)} \otimes \frake_{\Gamma_0(|D|)}
  \big|_{k, \rho_{\bbone_{|D|}} \otimes \rho_{\bbone_{|D|}}}\, \gamma
\end{gather*}
that was defined in Lemma~\ref{la:eisenstein-series-nonzero-at-one-cusp}. We can view $E_{k-l,\bbone_{|D|},\infty}$ as an Eisenstein series of type $\rho_{|D|} \otimes \rho_{M|D|^2}$ by means of the inclusion $\rho_{|D|} \hra \rho_{M |D|^2}$. We have
\begin{gather*}
  g
:=
  \pi\Big( \Ind\big( E_{l, \bbone_{|D|}} \big) \,\otimes\, \rmE_{k-l,\bbone_{|D|},\infty} \Big)
\in
  \rmE\rmE\big( \bbone_{|D|} \big)
\,\tx{,}
\\
  \pi :\,
  \rho_{|D|} \otimes ( \rho_{|D|} \otimes \rho_{M|D|^2} )
\lthra
  \rho_{M |D|^2}
\,\tx{;}
  \frake_\gamma \otimes \big( \frake_{\gamma'} \otimes \frake_{\gamma''} \big)
\lmto
  \begin{cases}
    \frake_{\gamma''}
  \,\tx{,}
  & \tx{if $\gamma = \gamma'$;}
  \\
     0
  \tx{,}
  & \tx{otherwise.}
  \end{cases}
\end{gather*}
In particular, by~\eqref{eq:proof-of-main-theorem:twist-is-orthogonal}, we have $\langle \Ind\, f_\epsilon,\, g \rangle = \langle f_\epsilon,\, \Ind^{-1}(g) \rangle = 0$ for this specific~$g$.

We now evalute $\langle \Ind\, f_\epsilon,\, g \rangle$ using the inclusion $\iota$ adjoint to $\pi$ to gain some flexibility with respect to the classical Eisenstein series $E_{l, \bbone_{|D|}}$.
\begin{gather*}
  \langle \Ind\, f_\epsilon,\, g \rangle
=
  \Big\langle
  \iota\big( \Ind\, f_\epsilon \big),\,
  \Ind\,E_{l, \bbone_{|D|}} \,\otimes\, E_{k-l,\bbone_{|D|},\infty}
  \Big\rangle
\tx{.}
\end{gather*}
We introduce a spectral parameter $s \in \CC$ for the Eisenstein series $E_{k-l,\bbone_{|D|},\infty}$, which we will later specialize to~$s = 0$. If $\Re\,s \gg 0$, then Proposition~\ref{prop:evaluation-of-pairing} computes the following Petersson scalar product in terms of a convolution Dirichlet series, and Proposition~\ref{prop:product-decomposition-of-convolution} allows us to factor it:
\begin{multline*}
  \Big\langle
  \iota\big( \Ind\, f_\epsilon \big),\,
  \Ind\,E_{l, \bbone_{|D|}} \,\otimes\, E_{k-l,\bbone_{|D|},\infty,\, s}
  \Big\rangle
\\
=
  (4 \pi)^{1-k}
  \sum_{n=1}^\infty n^{1-k + s}
  \sigma_{l-1, \bbone_{|D|}, \bbone}(n)\, c(f_{\epsilon_D};\, n)
=
  \frac{L\big(f_{\epsilon_D}, k-1 + s \big) L\big(f_{\epsilon_D} \times \bbone_{|D|},\, k-1 + s - (l-1) \big)}
       {L\big( \bbone_{|D|}\,\bbone_{|D|} \bbone_{M |D|^2},\, 2(k-1 + s) - (l-1) + 1 - k \big)}
\end{multline*}
The left and right hand side have analytic continuations to $s=0$, and from this we conclude that
\begin{gather}
\label{eq:proof-of-main-theorem:vanishing-L-function-product}
  L\big(f_{\epsilon_D}, k-1 \big)
  L\big(f_{\epsilon_D},\, k - l \big)
=
  0
\tx{.}
\end{gather}

The first factor cannot vanish, since $k - 1 > (k + 1) \slash 2$ lies in the convergent region of the Euler product for $L(f_{\epsilon_D}, \,\cdot\, )$. If $l < (k - 1) \slash 2$ then the second factor is nonzero by the same argument. This contradicts~\eqref{eq:proof-of-main-theorem:vanishing-L-function-product} and thus finishes the proof except if $(k - 1) \slash 2 \le l \le k \slash 2$, that is~$l = k \slash 2$. In the latter case we consider the central value of an $L$-function twisted by a imaginary quadratic character. We have to allude to Waldspurger's and Kohnen-Zagier's work~\cite{waldspurger-1981,kohnen-zagier-1984}. Their work on the Shimura correspondence shows that $L(f \times \epsilon_D, k \slash 2 )^{\frac{1}{2}}$ appears, up to normalizing factors, as the coefficient of a half-integral weight newform.

Since $D$ was an arbitrary but fixed negative fundamental discriminant, the theorem follows if we show that a newform~$h$ of half-integral weight
\begin{gather*}
  h(\tau)
=
  \sum_D c(D)\, q^D
\end{gather*}
vanishes if $c(D) = 0$ for all $D$ that are fundamental. By the Hecke theory for half-integral modular forms in Theorem~1.7 of~\cite{shimura-1973}, we find that $c(n^2 D) = 0$ for all $n \in \ZZ_{\ge 0}$. This finishes our proof of the main theorem.

\subsection{Example: Weight~\texpdf{$12$}{12}, level~\texpdf{$3$}{3}}

We illustrate the Theorem for weight~$12$ and level~$3$ modular forms.  Over~$\QQ$ we have $\Ind_{\Gamma_0(3)}\, \bbone \cong \bbone \oplus \rho_3$ for a three dimension representation. The second component has representation matrices
\begin{gather*}
  \rho_3(T)
=
  \begin{pmatrix}
  1 & 0 & 0 \\
  0 & 0 & 1 \\
  -1 & -1 & -1
  \end{pmatrix}
\,\text{,}\quad
  \rho_3(S)
  =
  \begin{pmatrix}
  0 & 1 & 0 \\
  1 & 0 & 0 \\
  -1 & -1 & -1
  \end{pmatrix}
\quad\tx{with respect to the basis}\quad
  \left(\begin{matrix} 1 \\ 0 \\ 0 \\ -1 \end{matrix}\right)
\tx{,}
  \left(\begin{matrix} 0 \\ 0 \\ 1 \\ -1 \end{matrix}\right)
\tx{,}
  \left(\begin{matrix} 0 \\ 1 \\ 0 \\ -1 \end{matrix}\right)
\tx{.}
\end{gather*}
Since the component $\bbone$ corresponds to level~$1$ modular forms, which are spanned by, say, $E_{12}$ and~$E_6^2$, it suffices to find a spanning set for $\rmM_{12}(\rho_3)$.

For $\rho_T := T_3\,\bbone$, we choose the basis
\begin{gather*}
  \frake_1
:=
  \frake_{\left(\begin{smallmatrix} 3 & 0 \\ 0 & 1 \end{smallmatrix}\right)}
\text{,}\,
  \frake_2
:=
  \frake_{\left(\begin{smallmatrix} 1 & 0 \\ 0 & 3 \end{smallmatrix}\right)}
\text{,}\,
  \frake_3
:=
  \frake_{\left(\begin{smallmatrix} 1 & 1 \\ 0 & 3 \end{smallmatrix}\right)}
\text{,}\,
  \frake_4
:=
  \frake_{\left(\begin{smallmatrix} 1 & 2 \\ 0 & 3 \end{smallmatrix}\right)}
\text{,}
\end{gather*}
so that
\begin{gather*}
  \rho_T(T)
=
  \begin{pmatrix}
  1 & 0 & 0 & 0 \\
  0 & 0 & 1 & 0 \\
  0 & 0 & 0 & 1 \\
  0 & 1 & 0 & 0
  \end{pmatrix}
\,\text{,}\quad
  \rho_T(S)
=
  \begin{pmatrix}
  0 & 1 & 0 & 0 \\
  1 & 0 & 0 & 0 \\
  0 & 0 & 0 & 1 \\
  0 & 0 & 1 & 0
  \end{pmatrix}
\,\text{.}
\end{gather*}

The tensor product $\rho_T \otimes \rho_T$ contains four copies of $\rho_3$.  The subspace of vectors that are fixed by $(\rho_T \otimes \rho_T)(T)$ is spanned by 
\begin{align*}
  \frakf_1
&
:=
  3 \frake_1 \otimes \frake_1
  - \frake_2 \otimes \frake_2 - \frake_3 \otimes \frake_3 - \frake_4 \otimes \frake_4
\,\text{,}
\\
  \frakf_2
&
:=
  \frake_1 \otimes \frake_2 + \frake_1 \otimes \frake_3
  + \frake_1 \otimes \frake_4 - \frake_2 \otimes \frake_4
  - \frake_3 \otimes \frake_2 - \frake_4 \otimes \frake_3
\,\text{,}
\\
  \frakf_3
&
:=
  \frake_2 \otimes \frake_1 - \frake_2 \otimes \frake_4
  + \frake_3 \otimes \frake_1 - \frake_3 \otimes \frake_2
  + \frake_4 \otimes \frake_1 - \frake_4 \otimes \frake_3
\,\text{,}
\\
  \frakf_4
&
:=
  \frake_2 \otimes \frake_3 - \frake_2 \otimes \frake_4
  - \frake_3 \otimes \frake_2 + \frake_3 \otimes \frake_4
  + \frake_4 \otimes \frake_2 - \frake_4 \otimes \frake_3
\,\text{.}
\end{align*}

The space $\rmM_{12}(\rho_3)$ has dimension $\dim\,\rmM_{12}(\Ind_{\Gamma_0(3)}) - \dim\,\rmM_{12} = 3$.  There are many possible ways to span $\rmS_{12}(\rho_3)$.  In order to emphasis the vector valued approach, we compute the components $\frakf_1$, $\frakf_2$, $\frakf_3$, and $\frakf_4$ of $(T_3\, E_4) \otimes (T_3\, E_8)$, where $E_4$ and $E_8$ are the level~$1$ Eisenstein series of weight~$4$ and $8$, respectively.  These components, one checks, are modular forms for $\Gamma_0(3)$ and determine the associated modular form in $\rmM_{12}(\rho_3)$ uniquely.
\begin{align*}
  f_1
:=
  \big\langle T_3 E_{4} \otimes T_3 E_{8},\, \frakf_1 \big\rangle
&=
  \tfrac{531440}{177147} - \tfrac{1883840}{19683}q - \tfrac{1274566720}{6561}q^{2} - \tfrac{330565225280}{19683}q^{3} - \tfrac{7831774435520}{19683}q^{4} + O\big( q^{5} \big)
\,\text{,}
\\[.5em]
  f_2
:=
  \big\langle T_3 E_{4} \otimes T_3 E_{8},\, \frakf_2 \big\rangle
&=
  \tfrac{80}{177147} + \big(-\tfrac{512000}{6561} \zeta_{3} + \tfrac{9449920}{19683}\big)q + \big(\tfrac{87040000}{2187} \zeta_{3} + \tfrac{774666560}{6561}\big)q^{2}
\\
&\qquad
  + \big(\tfrac{773632000}{729} \zeta_{3} + \tfrac{33711845440}{19683}\big)q^{3} + \big(\tfrac{14190592000}{6561} \zeta_{3} + \tfrac{122684877760}{19683}\big)q^{4} 
 + O\big( q^{5} \big)
\,\text{,}
\\[.5em]
  f_3
:=
  \big\langle T_3 E_{4} \otimes T_3 E_{8},\, \frakf_3 \big\rangle
&=
  \tfrac{6560}{177147} + \big(-\tfrac{512000}{6561} \zeta_{3} + \tfrac{4896640}{19683}\big)q + \big(\tfrac{87040000}{2187} \zeta_{3} + \tfrac{382920320}{6561}\big)q^{2}
\\
&\qquad
  + \big(\tfrac{773632000}{729} \zeta_{3} + \tfrac{13172759680}{19683}\big)q^{3} + \big(\tfrac{14190592000}{6561} \zeta_{3} - \tfrac{32958078080}{19683}\big)q^{4}
  + O\big( q^{5} \big)
\,\text{,}
\\[.5em]
  f_4
:=
  \big\langle T_3 E_{4} \otimes T_3 E_{8},\, \frakf_4 \big\rangle
&=
  \big(-\tfrac{1024000}{6561} \zeta_{3} - \tfrac{512000}{6561}\big)q + \big(\tfrac{174080000}{2187} \zeta_{3} + \tfrac{87040000}{2187}\big)q^{2}
\\
&\qquad
  + \big(\tfrac{1547264000}{729} \zeta_{3} + \tfrac{773632000}{729}\big)q^{3} + \big(\tfrac{28381184000}{6561} \zeta_{3} + \tfrac{14190592000}{6561}\big)q^{4}
  + O\big( q^{5} \big)
\,\text{.}
\end{align*}

We close this example by isolating the (unique) newform in $\rmM_{12}(\Gamma_0(3))$, which has initial Fourier expansion $q + 78 q^2 - 243 q^3 + 4036 q^4 + O(q^5)$. Solving for coefficients of $q^0$ through $q^3$ yields the expression
\begin{multline*}
 \big( \tfrac{2792336247}{93514815104000} \zeta_3 - \tfrac{7143127641}{187029630208000} \big) f_1
 +
 \big( \tfrac{118065190221}{93514815104000} \zeta_3 + \tfrac{2892599667}{187029630208000} \big) f_2
 \\
 +
 \big( \tfrac{227653108281}{93514815104000} \zeta_3 + \tfrac{144661293657}{46757407552000} \big) f_3
 +
 \big( \tfrac{227653108281}{187029630208000} \zeta_3 - \tfrac{144661293657}{93514815104000} \big) f_4
\,\tx{.}
\end{multline*}

\subsection{A curious vanishing condition}

As an immediate consequence of our Main Theorem, we get the following vanishing condition, for which at the moment, though, we have no application. For simplicity, we focus on the case $l < \frac{k}{2}$.
\begin{corollary}
\label{cor:strange-vanishing-condition}
Let $k$, $l$, and $\rho$ be as in the Main Theorem, and in addition assume that $l < \frac{k}{2}$. There is $0 < N_0 \in \ZZ$ such that if for $N, N' \isdiv N_0$
\begin{gather*}
  \rho_{N,\bbone}
=
  \big( \rho_N \otimes \rho_{N'} \otimes \rho^\vee \big)(\bbone)
\end{gather*}
has basis
\begin{gather*}
  \sum_{m_c, m_\sigma \in \Delta_N}
  \frake_{m_c} \otimes \frake_{m_\sigma} \otimes v_{N, i}(m_c, m_\sigma)
\,\txt{,}\quad
  1 \le i \le \dim\,\rho_{N,\bbone}
\end{gather*}
with $v_{N,i}(m_c,m_\sigma) \in V(\rho)^\vee$, then the following vanishing condition holds:  Given $f \in \rmM_k(\rho)$, then $f=0$ if for all $i$ and all $N \isdiv N_0$, we have
\begin{gather*}
  0
=
  \sum_{m_c,m_\sigma}
  a_c^{k-l} a_\sigma^l\;
  \sum_{n \in \QQ}
  e\big( \frac{- b_\sigma n}{a_\sigma} \big)\, \sigma_{l-1}\big( \frac{d_\sigma n}{a_\sigma} \big)\,
  \ov{v}_{N,i}(m_c,m_\sigma)\big( c(f;\,n) \big)
\txt{,}
\end{gather*}
where $m_c = \left(\begin{smallmatrix} a_c & b_c \\ c_c & d_c \end{smallmatrix}\right)$ and $m_\sigma = \left(\begin{smallmatrix} a_\sigma & b_\sigma \\ c_\sigma & d_\sigma \end{smallmatrix}\right)$.
\end{corollary}
\begin{proof}
This follows when computing the Fourier expansion of $T_M\, E$ for Eisenstein series~$E$ and then applying Proposition~\ref{prop:evaluation-of-pairing}.
\end{proof}


\renewbibmacro{in:}{}
\renewcommand{\bibfont}{\normalfont\small\raggedright}
\renewcommand{\baselinestretch}{.8}

\Needspace*{4em}
\begin{multicols}{2}
\printbibliography[heading=none]
\end{multicols}



\addvspace{1em}
\titlerule[0.15em]\addvspace{0.5em}

{\noindent\small
Chalmers tekniska högskola,
Institutionen för Matematiska vetenskaper,
SE-412 96 Göteborg, Sweden\\
E-mail: \url{martin@raum-brothers.eu}\\
Homepage: \url{http://raum-brothers.eu/martin}
}

\end{document}


%% file: packages.tex
\usepackage{ifdraft}

\usepackage[utf8]{inputenc}
\usepackage[T1]{fontenc}

\usepackage{lmodern}
\usepackage{fourier}

\usepackage{amssymb, amsfonts}

\usepackage{mathrsfs} 
\usepackage{dsfont}
\usepackage[usenames,dvipsnames]{xcolor}
\usepackage{lettrine}
\usepackage{url}

\usepackage{stmaryrd}

\usepackage[english]{babel}

\usepackage{scrpage2}
\usepackage[margin=7em]{geometry}
\usepackage[compact]{titlesec}
\usepackage{needspace}

\usepackage{amsmath}
\usepackage[thmmarks, amsmath, amsthm]{ntheorem}

\usepackage[all]{xy}

\usepackage[style=alphabetic,firstinits,url=false,sortcites=true,maxbibnames=99,backend=bibtex]{biblatex}
\bibliography{bibliography.bib}

\ifdraft{{}}
  {\usepackage[pdftex,
               pdfborder={0 0 0}, colorlinks=true,
               linkcolor=BrickRed, citecolor=ForestGreen, urlcolor=RoyalBlue]{hyperref}}

\usepackage{booktabs}
\usepackage{enumitem}
\usepackage{float}
\usepackage{multicol}

%% file: layout.tex
\ifdraft{\date{\today}}{\date{}}

\KOMAoptions{DIV=13}

\clearscrheadings
\automark[section]{subsection}

\renewcommand{\subsectionmark}[1]{}

\lohead{{\small \headertitle \;---\, \rightmark}}
\rohead{{\small \headerauthors}}

\newenvironment{plainfootnotes}{
  \deffootnote[0em]{0em}{0em}{}
}{
  \deffootnote[1em]{1.5em}{1em}{\textsuperscript{\thefootnotemark}}
}

\newif\ifsubsectionstylePrefixParagraph
\newif\ifsubsectionstyleRunin

\subsectionstylePrefixParagraphtrue
\subsectionstyleRunintrue

\titleformat{\section}[hang]
{\Large\sffamily\bfseries}
{\thesection\hspace{0.25em}{}}{0.25em}{}

\ifsubsectionstyleRunin
\ifsubsectionstylePrefixParagraph

\titleformat{\subsection}[runin]
{\normalfont\bfseries}
{\S\hspace{.1em}\thesubsection}{0.35em}{}[.\hspace*{.5em}]

\else

\titleformat{\subsection}[runin]
{\normalfont\bfseries}
{\thesubsection\hspace{.1em}|}{0.25em}{}[.\hspace*{.5em}]

\fi
\else
\ifsubsectionstylePrefixParagraph

\titleformat{\subsection}[wrap]
{\normalfont\bfseries\selectfont\filright}
{\S\thesubsection}{.35em}{}
\titlespacing{\subsection}
{16pc}{1.5ex plus .1ex minus .2ex}{1pc}

\else

\titleformat{\subsection}[wrap]
{\normalfont\bfseries\selectfont\filright}
{\thesubsection\hspace{.1em}|}{.25em}{}
\titlespacing{\subsection}
{16pc}{1.5ex plus .1ex minus .2ex}{1pc}

\fi
\fi

\ifsubsectionstylePrefixParagraph

\titleformat{\subsubsection}[runin]
{\normalfont\bfseries}
{\S\hspace{.1em}\thesubsubsection}{0.35em}{}[.]

\else

\titleformat{\subsubsection}[runin]
{\normalfont\bfseries}
{\thesubsubsection\hspace{.1em}|}{0.25em}{}[.]
  
\fi

\titleformat{\paragraph}[runin]
{\normalfont\itshape}
{\S\hspace{.1em}\theparagraph}{0.35em}{}[.]

\AtEveryBibitem{\clearfield{isbn}}
\AtEveryBibitem{\clearlist{language}}
\AtEveryBibitem{\clearfield{pages}}

\newenvironment{enumeratearabic}{
\begin{enumerate}[label=(\arabic*), leftmargin=0pt,labelindent=2em,itemindent=!]
}{
\end{enumerate}
}

\newenvironment{enumerateroman}{
\begin{enumerate}[label=(\roman*), leftmargin=0pt,labelindent=2em,itemindent=!]
}{
\end{enumerate}
}

\newenvironment{enumeratearabic*}{
\begin{enumerate*}[label=(\arabic*)] 
}{
\end{enumerate*}
}

\newenvironment{enumerateroman*}{
\begin{enumerate*}[label=(\roman*)] 
}{
\end{enumerate*}
}

\numberwithin{equation}{section}

\newtheorem{theoremcounter}{theoremcounter}[section]

\theoremstyle{plain}

\newtheorem{corollary}[theoremcounter]{Corollary}
\newtheorem{lemma}[theoremcounter]{Lemma}
\newtheorem{proposition}[theoremcounter]{Proposition}
\newtheorem{theorem}[theoremcounter]{Theorem}

\theoremnumbering{Roman}
\newtheorem{maintheoremcounter}{maintheoremcounter}

\newtheorem{maintheorem}[maintheoremcounter]{Theorem}

\newtheorem{mainremark}[maintheoremcounter]{Remark}

\newenvironment{mainremarkenumerate}
{
\begin{mainremark}
\begin{enumeratearabic}
}{
\end{enumeratearabic}
\end{mainremark}
}

\theoremstyle{definition}

\newtheorem{definition}[theoremcounter]{Definition}

\theoremstyle{remark}

\newtheorem{remark}[theoremcounter]{Remark}

\newtheorem*{remarkcomputation}{Computation}

\newenvironment{remarkenumerate}
{
\begin{remark}
\begin{enumeratearabic}
}{
\end{enumeratearabic}
\end{remark}
}



%% file: shortcuts.tex
\ifdraft{
 \newcommand{\texpdf}[2]{#1}
}{
 \newcommand{\texpdf}[2]{\texorpdfstring{#1}{#2}}
}

\newcommand{\tx}{\ensuremath{\text}}

\newcommand{\tbf}{\bfseries}

\newcommand{\thdash}{\nbd th}

\newcommand{\nbd}{\nobreakdash-\hspace{0pt}}

\renewcommand{\frak}{\ensuremath{\mathfrak}}

\newcommand{\frakc}{\ensuremath{\frak{c}}}

\newcommand{\frake}{\ensuremath{\frak{e}}}
\newcommand{\frakf}{\ensuremath{\frak{f}}}

\newcommand{\rmt}{\ensuremath{\mathrm{t}}}

\newcommand{\rmE}{\ensuremath{\mathrm{E}}}

\newcommand{\rmH}{\ensuremath{\mathrm{H}}}

\newcommand{\rmM}{\ensuremath{\mathrm{M}}}

\newcommand{\rmS}{\ensuremath{\mathrm{S}}}

\newcommand{\rmV}{\ensuremath{\mathrm{V}}}

\newcommand{\ov}{\overline}

\newcommand{\wht}{\widehat}

\newcommand*{\longhookrightarrow}{\ensuremath{\lhook\joinrel\relbar\joinrel\rightarrow}}
\newcommand*{\longtwoheadrightarrow}{\ensuremath{\relbar\joinrel\twoheadrightarrow}}

\newcommand{\ra}{\ensuremath{\rightarrow}}
\newcommand{\hra}{\ensuremath{\hookrightarrow}}
\newcommand{\thra}{\ensuremath{\twoheadrightarrow}}

\newcommand{\lra}{\ensuremath{\longrightarrow}}
\newcommand{\lhra}{\ensuremath{\longhookrightarrow}}
\newcommand{\lthra}{\ensuremath{\longtwoheadrightarrow}}

\newcommand{\mto}{\ensuremath{\mapsto}}
\newcommand{\lmto}{\ensuremath{\longmapsto}}

\newcommand{\amid}{\ensuremath{\mathop{\mid}}}

\newcommand{\ZZ}{\ensuremath{\mathbb{Z}}}
\newcommand{\QQ}{\ensuremath{\mathbb{Q}}}
\newcommand{\RR}{\ensuremath{\mathbb{R}}}
\newcommand{\CC}{\ensuremath{\mathbb{C}}}

\renewcommand{\Re}{\ensuremath{\mathfrak{Re}}}
\renewcommand{\Im}{\ensuremath{\mathfrak{Im}}}

\newcommand{\isdiv}{\amid}

\renewcommand{\pmod}[1]{\ensuremath{\;(\mathrm{mod}\, #1)}}

\newcommand{\Hom}{\ensuremath{\mathop{\mathrm{Hom}}}}

\newcommand{\GL}[1]{\ensuremath{\mathrm{GL}_{#1}}}

\newcommand{\SL}[1]{\ensuremath{\mathrm{SL}_{#1}}}

\newcommand{\T}{\ensuremath{\rmt}}
\newcommand{\rT}{\ensuremath{\,{}^\T\!}}

\newcommand{\diag}{\ensuremath{\mathrm{diag}}}

\newcommand{\bbone}{\ensuremath{\mathds{1}}}

\newcommand{\lspan}{\ensuremath{\mathop{\mathrm{span}}}}

\newcommand{\HS}{\mathbb{H}}



%% file: shortcuts_this.tex
\newcommand{\Ind}{\ensuremath{\mathrm{Ind}}}